\documentclass[10pt,reqno]{amsart}
\usepackage[nodisplayskipstretch]{setspace}
\usepackage[brazilian,english]{babel} 
\usepackage[T1]{fontenc}
\usepackage{amsfonts,amsthm,amssymb,amsmath,mathrsfs}  
\usepackage{graphicx,color}       
\usepackage{latexsym}       
\usepackage{overpic}        
\usepackage[colorlinks=false, linktocpage=true]{hyperref}
\usepackage{graphicx}
\usepackage{multicol}
\usepackage{color}
\usepackage{multirow} 
\usepackage{mathtools}
\usepackage{xcolor}
\mathtoolsset{showonlyrefs}

\newtheorem{theorem}{Theorem}[section]
\newtheorem{corollary}[theorem]{Corollary}

\newtheorem{lemma}[theorem]{Lemma}
\newtheorem{proposition}[theorem]{Proposition}

\theoremstyle{definition}
\newtheorem{definition}[theorem]{Definition}
\newtheorem{remark}[theorem]{Remark}

\newcommand{\E}{\mathcal{E}}

\newcommand{\R}{\mathbb{R}}
\newcommand{\N}{\mathbb{N}}
\newcommand{\C}{\mathbb{C}}

\newcommand{\Km}{\mathcal{K}}

\newcommand{\Dz}{\dot{\mathbf{H}}^{1}(\R^{d})}
\newcommand{\psib}{\mbox{\boldmath$\psi$}}
\newcommand{\phib}{\mbox{\boldmath$\phi$}}

\newcommand{\ub}{\mathbf{u}}
\newcommand{\uck}{u_{Ck}}
\newcommand{\vb}{\mathbf{v}}
\newcommand{\zb}{\mathbf{z}}

\newcommand{\wb}{\mathbf{w}}
\newcommand{\Lb}{\mathbf{L}}
\newcommand{\Wb}{\mathbf{W}}
\newcommand{\Sb}{\mathbf{S}}
\newcommand{\Xb}{\mathbf{X}}

\newcommand{\kl}{k=1,\ldots,l}
\newcommand{\sumk}{\sum_{k=1}^l}
\newcommand{\mR}{\mathcal{R}}
\usepackage{a4wide}
\usepackage{mathabx}

\newcommand{\RE}{\mathrm{Re}}
\newcommand{\IM}{\mathrm{Im}}

\numberwithin{equation}{section}

\begin{document}
	\title[NLS SYSTEM WITH $p$-TYPE NONLINEARITIES]{THE FOCUSING ENERGY-CRITICAL NONLINEAR SCHRÖDINGER SYSTEM WITH POWER-TYPE GROWTH NONLINEARITIES IN THE RADIAL CASE}
	
		\author{LUIZ GUSTAVO FARAH}\address{ICEx, Universidade Federal de Minas Gerais, Av. Antônio Carlos, 6627, Caixa Postal 702, 30123-970 Belo Horizonte, MG, Brazil}
	\email{lgfarah@gmail.com}
	
	\author{MAICON HESPANHA}\address{ICEx, Universidade Federal de Minas Gerais, Av. Antônio Carlos, 6627, Caixa Postal 702, 30123-970 Belo Horizonte, MG, Brazil}
	\email{mshespanha@gmail.com}

\subjclass[2020]{35Q44, 35P25, 35B44}
\keywords{Energy-critical; Nonlinear Schrödinger Equation; Ground State Solutions; Blow-up; Scattering }
	
	\begin{abstract} 
This work is concerned with a coupled system of focusing nonlinear Schrödinger equations involving general power-type nonlinearities in the energy-critical setting for dimensions $3 \leq d \leq 5$ in the radial setting. Our aim is to demonstrate the scattering versus blow-up dichotomy in the radial case. To achieve this, we first prove the existence of ground state solutions using the concentration-compactness method combined with variational techniques. We then establish finite-time blow-up through a convexity argument and prove scattering by applying the concentration-compactness and rigidity method.
	\end{abstract}
	
	\maketitle
	
	\section{Introduction}
	
	This paper investigates the Cauchy problem for the focusing energy-critical $l$-component nonlinear Schrödinger system
	\begin{equation}\label{SISTQ}
		\left\{\begin{array}{ll}
			i\alpha_k\partial_t u_k +\gamma_k \Delta u_k -\beta_k u_k=-f_k(\mathbf{u}), \quad  k=1,\ldots,l,\\
			\mathbf{u}(0,x)=\mathbf{u}_0(x),
		\end{array}\right.
	\end{equation}
	where $\mathbf{u}=(u_1,...,u_l)$ with $u_1,...,u_l$  complex-valued functions on the variables $(t,x)\in\mathbb{R}\times\mathbb{R}^d$, $d\geq3$, $\Delta$ is the Laplacian operator,  $\alpha_k,\gamma_k>0$, $\beta_k\geq 0$ are real constants. The nonlinearities $f_k: \mathbb{C}^l \to \mathbb{C}$ exhibit $p$-type growth and satisfy certain assumptions, which will be detailed later.
	
We are interested in the energy-critical case, where the ${\dot{H}}^{1}(\R^d)$-norm of the solutions to \eqref{SISTQ} are invariant under the scaling symmetry
	\begin{equation*}\label{scaling}
		g(t,x)\mapsto g_\lambda(t,x):=\lambda^{\frac{d-2}{2}}g(\lambda^2t,\lambda x),\quad (t,x)\in\mathbb{R}\times\mathbb{R}^d.
	\end{equation*}
A direct calculation shows that $p = \frac{d+2}{d-2}$, where $p$ denotes the power of the nonlinearity $f_k$, $k=1,...,l$. In particular, for dimensions $d=3$, $d=4$ and $d=6$, the nonlinearities are quintic, cubic and quadratic, respectively.
	
	In recent years, there has been growing interest in studying the dynamics of nonlinear coupled Schrödinger systems with nonlinearities exhibiting power-like growth. For instance, in the specific case where $l =p= 2$, $f_1(u_1,u_2)=2\overline{u}_1u_2$ and $f_2(u_1,u_2)={u}^2_1$, the system is given by
	\begin{equation}\label{system1J}
		\begin{cases}
			\displaystyle i\partial_{t}u_1+\Delta u_1=-2\overline{u}_1u_2,\\
			\displaystyle i\partial_{t}u_2+\kappa\Delta u_2=- u^{2}_1,
		\end{cases}
	\end{equation}
where $\kappa>0$.  This system arises as the non-relativistic limit of a nonlinear Klein–Gordon system, see \cite{Hayashi2013onsystem} for further details.


The work \cite{Hayashi2013onsystem} also addresses the local and global well-posedness and the existence of \textit{ground state} solutions for \eqref{system1J}. Following this, several results concerning the dynamics of solutions to \eqref{system1J} have been established, including existence of bound and ground states in dimensions $d=2,3$ \cite{zhang2018stable}; scattering in dimension $d=4$ \cite{inui2019scattering}; global versus blow-up dichotomy in dimension $d=5$ \cite{NoPa}; existence and stability of standing waves in dimensions $1\leq d\leq 3$ and finite-time blow-up solutions in dimension $d=4$ \cite{dinh2020existence}; strong instability of standing waves in dimension $d=5$ \cite{dinh2020instability}; finite-time blow-up in dimensions $d=5,6$ and blow-up or grow-up in dimension $d=4$ \cite{inui}; and existence of traveling wave solutions in dimensions $1\leq d\leq 5$ \cite{Fukaya2022travel}.

The study of nonlinear Schrödinger systems with quadratic interactions is not limited to two-wave interactions. In particular, models exhibiting quadratic-type nonlinear growth for the case $l=3$ have been investigated in \cite{pastor}, \cite{uriya2016finalstate}, and \cite{meng2021scattering}. The energy-critical case, where $d=6$, was addressed in \cite{inui}, where, among other results, finite-time blow-up was shown under radial symmetry assumptions. Additionally, \cite{gao} established a complete classification of the scattering versus blow-up dichotomy based on whether the initial kinetic energy is below or above that of the ground state.
	
	Another example, where $l=2$ and $p=3$, is the following
	\begin{equation}\label{SISTcubic}
		\begin{cases}
			i\partial_tu_1+\Delta u_1-u_1+\left(\dfrac{1}{9}|u_1|^2+2|u_2|^2\right)u_1+\dfrac{1}{3}\bar{u_1}^2u_2=0,\\
			i\sigma \partial_tu_2+\Delta u_2-\mu u_2+(9|u_2|^2+2|u_1|^2)u_2+\dfrac{1}{9}u_1^3=0,\\
		\end{cases}
	\end{equation}
where $\sigma,\mu$ are positive real constants. This model describes the interaction between an optical beam and its third harmonic in a material with Kerr-type nonlinear response, for a more details we refer the reader to \cite{SBK}.

From a mathematical point of view, system \eqref{SISTcubic} has been studied in various settings. In \cite{AP} and \cite{pastor4}, the authors established local and global well-posedness in Sobolev spaces of positive index for periodic initial data in one spatial dimension. They also proved the existence of periodic standing waves of dnoidal and cnoidal type, as well as their spectral and nonlinear stability in the energy space. For the multidimensional case in $\R^d$, with $1\leq d\leq3$, several results concerning the dynamics of the solutions to system \eqref{SISTcubic} were established in \cite{oliveira}. These include the existence and stability of ground state solutions, local and global well-posedness in the energy space $H^1(\mathbb{R}^d)\times H^1(\mathbb{R}^d)$ and finite-time blow-up criteria. Moreover, the energy-critical case $d=4$ was recently studied in \cite{HP}, where it was shown that radially symmetric solutions with initial energy below that of the ground state, but with kinetic energy exceeding that of the ground state, must blow up in finite time. Additional results on system \eqref{SISTcubic} can be found in \cite{stefanov}, \cite{colin}, \cite{zhang}, and \cite{likai}.
	
Motivated by \cite{NoPa}-\cite{pastor3}, we aim to provide sufficient conditions on the interaction terms $f_k$ to study the dynamics of energy-critical systems, $p=\frac{d+2}{d-2}$, in general dimension $d\geq3$. This research extends the subcritical framework, $\frac{d+4}{d}\leq p<\frac{d+2}{d-2}$ , established in \cite{noguerap}. To this end, we will impose the following assumptions on the nonlinear terms
	
	\newtheorem{thmx}{}
	\renewcommand\thethmx{(H1)}
	\begin{thmx}\label{H1}
		The nonlinearities satisfy
		\begin{align*}
			f_{k}(\mathbf{0})=0, \qquad  k=1,\ldots,l.
		\end{align*}
	\end{thmx}
	
	\renewcommand\thethmx{(H2)}
	\begin{thmx}\label{H2}
		We have, for any $\mathbf{z},\mathbf{z}'\in \C^{l}$, that
		\begin{equation*}
			\begin{split}
				\left|\frac{\partial }{\partial z_{m}}[f_{k}(\mathbf{z})-f_{k}(\mathbf{z}')]\right|+ \left|\frac{\partial }{\partial \overline{z}_{m}}[f_{k}(\mathbf{z})-f_{k}(\mathbf{z}')]\right|&\leq C \sum_{j=1}^{l}|z_{j}-z_{j}'|^{\frac{4}{d-2}},\qquad k,m=1,\ldots,l.
			\end{split}
		\end{equation*}
	\end{thmx}

	\renewcommand\thethmx{(H3)}
	\begin{thmx}\label{H3}
		There exists a function $F:\C^{l}\to \C$,  such that
		\begin{equation*}
			f_{k}(\mathbf{z})=\frac{\partial F}{\partial \overline{z}_{k}}(\mathbf{z})+\overline{\frac{\partial F }{\partial z_{k}}}(\mathbf{z}),\qquad k=1,\ldots,l.
		\end{equation*}
	\end{thmx}
	
	\renewcommand\thethmx{(H4)}
	\begin{thmx}\label{H4}
		There exists positive constants $\sigma_1,...,\sigma_l$ such that for any $\zb\in\C^l$
		$$
		Im\sum_{k=1}^l\sigma_kf_k(\zb)\widebar{z}_k=0,
		$$
		
	\end{thmx}
	
	\renewcommand\thethmx{(H5)}
	\begin{thmx}\label{H5}
		The function $F$ is homogeneous of degree $\frac{2d}{d-2}$, that is, for any $\mathbf{z}\in \mathbb{C}^{l}$ and $\lambda >0$,
		\begin{equation*}
			F(\lambda \mathbf{z})=\lambda^{\frac{2d}{d-2}}F(\mathbf{z}).
		\end{equation*}
	\end{thmx}

	\renewcommand\thethmx{(H6)}
	\begin{thmx}\label{H6}
		The function $F$ satisfies the following inequality
		\begin{equation*}
			\left|\mathrm{Re}\int_{\R^{d}} F(\ub)\;dx\right|\leq \int_{\R^{d}} F(\!\!\big\bracevert\!\! \mathbf{u}\!\!\big\bracevert\!\!)\;dx.
		\end{equation*}
	\end{thmx}

	\renewcommand\thethmx{(H7)}
	\begin{thmx}\label{H7}
		The function $F$ is real valued on $\R^l$, that is, if $(y_{1},\ldots,y_{l})\in \R^{l}$ then
		\begin{equation*}
			F(y_{1},\ldots,y_{l})\in \R.
		\end{equation*}
		Moreover, all the functions	$f_k$ are non-negative on the positive cone in $\mathbb{R}^l$, that is, for $y_i\geq0$, $i=1,\ldots,l$,
		\begin{equation*}
			f_{k}(y_{1},\ldots,y_{l})\geq 0, \qquad  k=1,\ldots,l.
		\end{equation*}	
	\end{thmx}

	\renewcommand\thethmx{(H8)}
	\begin{thmx}\label{H8}
		The function $F$ can be written as the sum $F=F_1+\cdots+F_m$, where $F_s$, $s=1,\ldots, m$ is super-modular on $\R^d_+$, $1\leq d\leq l$ and vanishes on hyperplanes, that is, for any $i,j\in\{1,\ldots,d\}$, $i\neq j$ and $k,h>0$, we have
		\begin{equation*}
			F_s(y+he_i+ke_j)+F_s(y)\geq F_s(y+he_i)+F_s(y+ke_j), \qquad y\in \R^d_+,
		\end{equation*}
		and $F_s(y_1,\ldots,y_d)=0$ if $y_j=0$ for some $j\in\{1,\ldots,d\}$.
	\end{thmx}
	
As we will show, assumptions \ref{H1}–\ref{H2} suffice to prove local well-posedness in the energy space (see Theorem \ref{LWP} below). Moreover, assumptions \ref{H3}-\ref{H4} guarantee that solutions to system \eqref{SISTQ} conserve both mass and energy, which are 
given, respectively, by
	\begin{equation}\label{qmassa}
		M(\mathbf{u}(t)):=\sum_{k=1}^l\frac{\sigma_k\alpha_k}{2}\Vert u_k(t)\Vert_{L^2}^2
	\end{equation}
	and
	\begin{equation}\label{EQ}
		E(\mathbf{u}(t)):=K(\ub)+L(\ub)-2P(\ub),
	\end{equation}
	where $K$, $L$ and $P$ are defined as
	\begin{equation}\label{KLP}
		K(\mathbf{u})= \sum_{k=1}^l\gamma_k\Vert\nabla u_k\Vert_{L^2}^2,\quad L(\ub)=\sumk \beta_k\Vert u_k\Vert_{L^2}^2 \quad \hbox{and}\quad P(\mathbf{u})=\hbox{Re}\int_{\mathbb{R}^d} F(\mathbf{u}(t))dx.
	\end{equation}
The functionals $K$ and $P$ are referred to as the kinetic and potential energies, respectively.
	
To study the asymptotic behavior of the global solutions to the energy-critical system \eqref{SISTQ}, under assumptions \ref{H1}-\ref{H8}, we well also impose a mass resonance condition. More precisely, we say that the system satisfies the mass-resonance condition if
	\begin{equation}\label{RC}
		\mathrm{Im}\sum_{k=1}^{l}\frac{\alpha_{k}}{2\gamma_{k}}f_{k}(\zb)\overline{z}_{k}=0, \quad \zb\in \mathbb{C}^{l},
	\end{equation}
where $\alpha_k$ and $\gamma_k$ are the parameters appearing in \eqref{SISTQ}.

This condition simplifies the application of localized virial-type arguments. Indeed, let $\phi\in C_0^\infty(\R^d)$ and define
\begin{equation}\label{Vt}
		V(t)=\int \phi(x)\left(\sum_{k=1}^l\frac{\alpha_{k}^2}{\gamma_k}|u_k|^2\right)dx.
	\end{equation}
Taking the time derivative, we obtain
	\begin{equation}\label{V'}
		V'(t)=\sum_{k=1}^l \IM\int\nabla\phi\cdot\nabla u_k\widebar{u}_kdx-4\int\phi(x)\IM\sum_{k=1}^l\frac{\alpha_{k}}{2\gamma_k}f_k(\ub)\widebar{u}_kdx,
	\end{equation}
where the last term vanishes due to the mass-resonance condition \eqref{RC}.
	
	\begin{remark}
		Note that systems \eqref{system1J} and \eqref{SISTcubic} satisfy assumptions \ref{H1}-\ref{H8}. Indeed, their associated potential functions are given respectively by $F(z_1,z_2)=\widebar{z}_1^2z_2$ and $F(z_1,z_2)=\frac{1}{36}|z_1|^4+\frac{9}{4}|z_2|^4+|z_1|^2|z_2|^2+\frac{1}{9}\widebar{z}_1^3z_2$. Moreover, using the mass-resonance condition \eqref{RC}, one can verify that both systems satisfy this condition if, and only if, $\kappa=\frac{1}{2}$ and $\sigma=3$, respectively.
	\end{remark}
Our main goal is to establish a classification of the scattering versus blow-up dichotomy in the radial setting.  Before presenting our results, we first recall the known results for the Cauchy problem associated with the single nonlinear Schrödinger equation		
	\begin{equation}\label{NLS}
		\left\{
		\begin{array}{ll}
			i\partial_t u+\Delta u=\lambda |u|^{\frac{4}{d-2}}u, \quad (t,x)\in\R \times \R^d,\\
			u(0,x)=u_0.
		\end{array}
		\right.
	\end{equation}
The local well-posedness was first established by \cite{CAZW}. Furthermore, to investigate the asymptotic behavior of radial solutions, the authors in the seminal work \cite{KM} developed the concentration-compactness and rigidity method, proving scattering below the energy threshold for radial initial data in dimensions $3\leq d\leq 5$. These results were later extended to the non-radial setting in dimensions $d\geq 5$ \cite{KV}, and more recently in dimension $d=4$ \cite{D}.
	
	Turning back to our problem, we start with the following definition. 
	\begin{definition}[Solution]\label{Solution} A solution to the system \eqref{SISTQ} is a function $\mathbf{u}:I\times\mathbb{R}^d\to \C^{l}$, defined on a non-empty time interval $ I\subset \mathbb{R}$, with $0\in I$, such that $\mathbf{u} \in \mathbf{C}_t^0\dot{\mathbf{H}}_x^1(K\times\mathbb{R}^d)\cap\mathbf{L}_{t,x}^d(K\times\mathbb{R}^d)$ for all compact interval $K\subset I$, and satisfies the Duhamel formula
		\begin{equation}\label{defsol}
			\left\{\begin{array}{ll}
				u_k(t)=U_k(t)u_{k0}+i\int_0^tU_k(t-s)\frac{1}{\alpha_k}f_k(\mathbf{u})ds,\\
				\mathbf{u}(0)=\mathbf{u}_0=(u_{10},...,u_{l0}),
			\end{array}\right.
		\end{equation}
where $U_k(t)$ denotes the corresponding unitary group defined by $U_k(t)=e^{it\frac{1}{\alpha_k}(\gamma_k\Delta-\beta_k)}$, $k=1,...,l$, and $t\in I$. The interval $I$ is referred to as the lifespan of the solution $\mathbf{u}$. We say that $\mathbf{u}$ is a maximal-lifespan solution if the solution cannot be extended to an interval $J\supset I$ strictly larger than $I$. The solution is said to be global if $I=\mathbb{R}$. 
	\end{definition}


To state our local well-posed result we introduce the norms $S(I)$ and $X(I)$ for an interval $I$ as follows
		\begin{equation}\label{norms2}
			\Vert \cdot \Vert_{S(I)}:=\Vert \cdot \Vert_{L_I^{\frac{2(d+2)}{d-2}}L_x^{\frac{2(d+2)}{d-2}}}\,\hbox{ and }\	\Vert \cdot \Vert_{X(I)}:=\Vert \cdot \Vert_{L_I^{\frac{2(d+2)}{d-2}}L_x^{\frac{2d(d+2)}{d^2+4}}}.
		\end{equation}
Note that $(\frac{2(d+2)}{d-2}, \frac{2d(d+2)}{d^2+4})$ is an admissible pair as defined in Definition \ref{Adpair} below. Furthermore, by the Sobolev embedding (see Lemma \ref{sobolev} below), it follows that 
		\begin{equation}\label{SXnorms}
		\Vert \mathbf{v}\Vert_{\Sb(I)}\lesssim \Vert \nabla \mathbf{v}\Vert_{\Xb (I)}.
		\end{equation}
By following the strategy for the energy critical Schr\"odinger equation \eqref{NLS} (see, for example, \cite[Theorem 1.1]{CAZW} or \cite[Theorem 2.5]{KM}, and for a textbook exposition, \cite[Chapter 4]{cazenave} or \cite[Chapter 5]{linares}), we can establish the following theorem.		
	\begin{theorem}\label{LWP}
		Let $3 \leq d \leq 5$ and suppose that assumptions \ref{H1}-\ref{H2} are satisfied. For $\mathbf{u}_0 \in \Dz$, there exists $T = T(\ub_0) > 0$ such that for the interval $I = [-T(\ub_0), T(\ub_0)]$, there is a unique solution $\ub \in C(I; \dot{\mathbf{H}}^1(\R^d))$ to the Cauchy problem \eqref{SISTQ}, with $\ub \in \mathbf{S}(I)$ and $\nabla \ub \in \mathbf{X}(I)$. 
	\end{theorem}
	
	It is well known that the ground state solution to the single nonlinear Schrödinger equation is closely linked to the sharp constant in the critical Sobolev inequality (see, for example, \cite{T}). Our first goal is to is to investigate the existence of a special solution for system \eqref{SISTQ}. More precisely, a {standing wave} solution for \eqref{SISTQ} is a particular solution of the form
	\begin{equation}\label{SWS}
		u_k(t,x)=e^{i\frac{\sigma_k}{2}\omega t}\psi_k(x),\quad k=1,...,l.,
	\end{equation}
	where $\omega\in\R$ and $\psi_k$ are real-valued functions decaying to zero at infinity. Substituting into \eqref{SISTQ}, we find that $\psi_k$ satisfies the equation
	\begin{equation}\label{SISTe}
		-\gamma_k\Delta\psi_k +b_k\psi_k=f_k(\psib),\quad k=1,\ldots,l,
	\end{equation}
	where $b_k:=\frac{\alpha_k^2}{\gamma_k}\omega+\beta_k\geq 0$. The action function associated with \eqref{SISTe} is
	\begin{equation}\label{action}
		S(\boldsymbol{\psi})=\frac{1}{2}\left[\sum_{k=1}^l \int\gamma_k|\nabla \psi_k|^2+\sum_{k=1}^l \int b_k|\psi_k|^2 \right]-\int F(\psib)dx.
	\end{equation}
	A ground state solution for \eqref{SIST3} is a nontrivial critical point of the functional $S$ that also minimizes it. According to \cite[Lemma 4.3]{noguerap}, nontrivial solutions to \eqref{SISTe} exist only if $\omega = 0$ and $\beta_k = 0$ for $k=1,\ldots,l$ (and thus $b_k = 0$). Consequently, equation \eqref{SISTe} simplify to
	\begin{equation}\label{SIST3}
		-\gamma_k\Delta \psi_k=f_k(\boldsymbol{\psi}), \quad k=1,\ldots,l,
	\end{equation}
with the corresponding action functional expressed as
	\begin{equation}\label{action2}
		S(\boldsymbol{\psi})=\frac{1}{2}K(\psib)-P(\psib).
	\end{equation}
The preceding discussion leads to the following definition.
	\begin{definition}
A function $\psib\in\Dz$ is a \textit{weak solution} of \eqref{SIST3} if for any $\mathbf{g}\in\Dz$, 
			\begin{equation}\label{solfraca}
				\gamma_k\int\nabla \psi_k \cdot \nabla g_kdx=\int f_k(\psib)g_kdx,\quad k=1,\ldots,l.
			\end{equation}
Moreover, a weak solution $\psib\in\Dz$ is a \textit{ground state} of \eqref{SIST3} if
			$$
			S(\psib)=\inf\{S(\phib);\,\,\phib\in\mathcal{C}\},
			$$
			where where $\mathcal{C}$ denotes the set of all non-trivial solutions of  \eqref{SIST3}.  The set of all ground state of \eqref{SIST3} will be  denote by $\mathcal{G}$.
	\end{definition}
The existence of this special solution is established in the following theorem.
	\begin{theorem}\label{ESGS}
		The set $\mathcal{G}$ is non-empty, and thus there exists at least one ground state solution $\boldsymbol{\psi}$ for system \eqref{SIST3}.
	\end{theorem}		

The proof of previous theorem is inspired by the recent work of the second author \cite{HP}. We employ the concentration-compactness principle from \cite{lions2}, constructing a sequence of probability Radon measures and excluding the possibilities of vanishing and dichotomy to establish compactness. As a consequence, we also obtain the optimal constant for some critical Sobolev inequality (see Corollary \ref{corol314} below).
	
Once the existence of ground state solutions is established, we can set the threshold for our dichotomy. We begin by studying the global existence and scattering of radial solutions. In this part, we work with initial data in the homogeneous Sobolev space $\Dz$, which means that we do not have control over the $L^2$ norm. Consequently, as in the analysis of ground state solutions, we set all $\beta_k = 0$ in \eqref{SISTQ} and instead work with the system
	\begin{equation}\label{SISTR}
		\left\{\begin{array}{ll}
			i\alpha_k\partial_t u_k +\gamma_k \Delta u_k=-f_k(\mathbf{u}), \quad k=1,...,l,\\
			\mathbf{u}(0,x)=\mathbf{u}_0(x).
		\end{array}\right.
	\end{equation}
Note that in this case, the energy is given by
\begin{equation}\label{anotherE}
\E(\mathbf{u}(t))=\E(\mathbf{u}_0):=K(\mathbf{u}_0)-2P(\mathbf{u}_0).
\end{equation}
Our goal is to prove that, under the mass-resonance condition \eqref{RC}, if the initial data has both energy and kinetic energy strictly below those of a ground state, then the corresponding solution to \eqref{SISTR} is global and scatters. To this end, we adopt the Kenig-Merle approach introduced in \cite{KM} on the radial setting. Our result is stated as follows.
	\begin{theorem}\label{MAIN}
Let $3 \leq d \leq 5$ and assume that \ref{H1}-\ref{H8} and the mass-resonance condition \eqref{RC} hold. If $\mathbf{u}_0\in \Dz$ is radially symmetric and $\boldsymbol{\psi}$ is any ground state solution satisfying
		\begin{equation}\label{mainE}
			\E(\mathbf{u}_0)<\E(\boldsymbol{\psi})
		\end{equation}
		and
		\begin{equation}\label{main K}
			K(\mathbf{u}_0)<K(\boldsymbol{\psi}),
		\end{equation}
then the corresponding solution $\mathbf{u}$ with initial data $u_0$ at $t=0$ for the system \eqref{SISTR} is global. Moreover, there exists $\mathbf{u}^+,\mathbf{u}^-\in\dot{\mathbf{H}}^1$ such that
	$$
	\lim_{t\rightarrow\pm} \Vert u_k(t)-u_k^{\pm}\Vert_{\dot{\mathbf{H}}^1}=0, \quad k=1,...,l.
	$$
	\end{theorem}
The above result may be extended to $d \geq 5$ without the radial assumption by employing the strategy in \cite{KV}. This will be explored in future work.
	
Finally, we investigate blow-up in finite time. Here, we consider initial data in $\mathbf{H}^1(\mathbb{R}^d)$, which gives us control over the $L^2$ norm. Next we state the local theory in this space.
\begin{theorem}\label{LWP2}
Let $3 \leq d \leq 5$ and suppose that assumptions \ref{H1}-\ref{H2} are satisfied. For $\mathbf{u}_0 \in \mathbf{H}^1(\mathbb{R}^d)$, there exists $T = T(\ub_0) > 0$ such that for the interval $I = [-T(\ub_0), T(\ub_0)]$, there is a unique solution $\ub \in C(I; {\mathbf{H}}^1(\R^d))$ to the Cauchy problem \eqref{SISTQ}, with $\ub \in \mathbf{S}(I)\cap \mathbf{X}(I)$ and $\nabla \ub \in \mathbf{X}(I)$.
\end{theorem}
Since we can also work with the $L^2$ norm, there is no need to impose the mass-resonance condition \eqref{RC}, and we can work directly with the original system \eqref{SISTQ}. The following result asserts that if the initial data has energy below that of the ground states, but kinetic energy exceeding that of the ground states, then the maximal interval of existence of the corresponding solution must be finite.
	\begin{theorem}\label{blow4}
		Let $3 \leq d \leq 5$. Suppose $\mathbf{u}_0\in \mathbf{H}^1(\mathbb{R}^d)$ and let $\mathbf{u}$ be the corresponding solution of \eqref{SISTQ} defined in the maximal time interval of existence, say, $I$. If $\mathbf{u}_0$  is  radially symmetric functions satisfying
		\begin{equation}\label{41}
			E(\mathbf{u}_0)<\mathcal{E}(\boldsymbol{\psi})
		\end{equation}
		and
		\begin{equation}\label{42}
			K(\mathbf{u}_0)>K(\boldsymbol{\psi}),
		\end{equation}
		where $\boldsymbol{\psi}$ is any ground state solution, and $\mathcal{E}$ is the energy defined in \eqref{anotherE}, then the time interval $I$ is finite.
	\end{theorem} 
The condition $3 \leq d \leq 5$ in the above result is required only due to Theorem \ref{LWP2}, though the proof method applies to all dimensions $d \geq 3$. To prove the above theorem, we employ the convexity method, which, roughly speaking, consists of deriving a contradiction by working with the virial identities for the function \eqref{Vt}. Note that the second term in the identity \eqref{V'} vanishes only under the mass-resonance assumption \eqref{RC}. To address this, we adopt a modification of the method, as in \cite{inui}, which involves discarding the second term in \eqref{V'} and instead working with the function
	\begin{equation}\label{R}
		\mathcal{R}(t)=\sum_{k=1}^l 2\IM\int\nabla\phi\cdot\nabla u_k\widebar{u}_kdx.
	\end{equation}

	The paper is organized as follows. In Section \ref{secnot}, we introduce the basic notation and preliminary results that will be used throughout the work. Section \ref{secgs} is devoted to establishing the existence of ground state solutions. The global well-posedness and scattering results are presented in Section \ref{secgwp}. Finally, the blow-up result stated in Theorem \ref{blow4} is proved in Section \ref{secbl}.

	\section{Notation and preliminary results}\label{secnot}
	\subsection{Notation}
In this subsection we introduce some notation used throughout the paper. The letter $C$ will be used to denote a positive constant that may vary from line to line. If necessary, we indicate different parameters the constant may depends on. Given any positive quantities $a$ and $b$, the notation $a \lesssim b$ means
that $a \leq Cb$, with $C$ uniform with respect to the set where $a$ and $b$ vary. 
We use $a+$  to denote $a+\varepsilon$ for a sufficiently small $\varepsilon>0$. Here $B(x,R)$ denotes the ball in $\R^d$ centered at $x$ with radius $R>0$. For a complex number $z\in \mathbb{C}$, Re$\,z$ and Im$\,z$ represents its real and imaginary parts.  Also, $\bar{z}$ denotes the complex conjugate of $z$. We denote by $\mathbf{z}\in\C^l$ the vector $(z_1,...,z_l)$, where $z_m=x_m+iy_m$ with $x_m,z_m$ the imaginary parts of $z_m$. We set $\big\bracevert\! \mathbf{z}\!\big\bracevert$ for the vector $(|z_1|,\ldots,|z_l|)$. This is not to be confused with $|\mathbf{z}|=\sqrt{|z_1|^2+\cdots+|z_l|^2}$ which denotes the usual norm of the vector $\mathbf{z}\in\mathbb{C}^l$. The notation $\mathbf{z} \geq 0$ indicates that $z_j \geq 0$ for all $j = 1, \dots, l$.

As usual,  the operators $\partial/\partial z_m$ and $\partial/\partial \widebar{z}_m$ are defined by
	$$
	\frac{\partial}{\partial z_m}=\frac{1}{2}\left(\frac{\partial}{\partial x_m}-i\frac{\partial}{\partial} y_m\right),\quad \frac{\partial}{\partial \widebar{z}_m}=\frac{1}{2}\left(\frac{\partial}{\partial x_m}+i\frac{\partial}{\partial} y_m\right).
	$$
For a subset $A$, we denote by $\mathbf{A}$ the product $A \times \cdots \times A$ ($l$-times). In particular, if $A$ is a Banach space, then $\mathbf{A}$ is also a Banach space equipped with the standard norm defined by the sum. We denote the standard Sobolev, homogeneous Sobolev, and Lebesgue spaces by $H^{s,p}=H^{s,p}(\mathbb{R}^d)$, $\dot{H}^{s,p}=\dot{H}^{s,p}(\R^d)$ and $L^p=L^p(\mathbb{R}^d)$, respectively, equipped with their standard norms. For simplicity, we use $H^s=H^{s,2}$ and $\dot{H}^s=\dot{H}^{s,2}$. For a given time interval $I$, the mixed Lebesgue space $L_t^pL_x^q(I\times\mathbb{R}^d)$, denoted by $L_t^pL_x^q$, is endowed with the norm
	$$
	\Vert f \Vert_{L_t^pL_x^q}=\left(\int_I\left(\int_{\mathbb{R}^d}|f(t,x)|^qdx\right)^{p/q}dt\right)^{1/p},
	$$
with the obvious modification if either $p=\infty$ or $q=\infty$.

	
	\subsection{Implications of the assumptions \ref{H1}–\ref{H8}}
Here we present some properties derived from the assumptions imposed on the nonlinearities $f_k$. First, we prove that mass and energy are conserved quantities.
	\begin{lemma}\label{Conserv}
		Assume that \ref{H3} and \ref{H4} hold. Let $\ub_0\in\mathbf{H}^1(\mathbb{R}^d)$ and $u(t)$ be the corresponding solution to \eqref{SISTQ}, then
		$$
		M(\ub(t))=M(\ub_0)\quad\hbox{and}\quad E(\ub(t))=E(\ub_0),
		$$
		where $M$ and $E$ are defined in \eqref{qmassa} and \eqref{EQ}.
	\end{lemma}
	\begin{proof}
		First, observe that multiplying \eqref{SISTQ} by $u_k$ we get
		$$
		i\alpha_k\partial_t|u_k|^2+\gamma_k\Delta u_k\cdot \bar{u}_k-\beta_k|u_k|^2=f_k(\ub)\bar{u}_k.
		$$
By integrating by parts with respect to $x$, taking the imaginary part of the result, summing over $k$, and applying \ref{H4}, we obtain
		$$
		\partial_t\left(\sumk \frac{\sigma_k\alpha_k}{2}\Vert u_k(t)\Vert_{L^2}^2\right)=-2\IM \sumk \int \sigma_kf_k(\ub)\bar{u}_kdx=0.
		$$
		Therefore, the mass of the system is a conserved quantity. We multiply \eqref{SISTQ} by $\partial_t\bar{u}_k$, add its complex conjugate, integrate by parts with respect to $x$, and sum over $k$ to deduce
		\begin{equation}\label{conservenergy}
			\partial_t\left(\sumk \gamma_k\Vert\nabla u_k\Vert_{L^2}^2+\sumk \beta_k \Vert u_k\Vert_{L^2}^2\right)=2\RE\int\left(\sumk f_k(\ub)\partial_t\bar{u}_k\right)=2\RE\int \partial_t F(\ub(t))dx,
		\end{equation}
where the last equality follows from \ref{H3} and the chain rule. Combining \eqref{conservenergy} with \eqref{EQ}, we get
		$$
		\partial_t E(\ub(t))=0,
		$$
which implies the energy conservation.	
	\end{proof}
		In the following two lemmas, we establish estimates for the nonlinearities $f_k$ and the potential function $F$ appearing in \ref{H3}.
	\begin{lemma}\label{consequences} Let \ref{H1} and \ref{H2} hold.
		\begin{itemize}
			\item[(i)] For any $\mathbf{z},\mathbf{z}'\in \mathbb{C}^l$, we have
			\begin{equation}\label{fk21}
				|f_k(\mathbf{z})-f_k(\mathbf{z}')|\lesssim \sum_{m=1}^l\sum_{j=1}^l\left(|z_j|^{\frac{4}{d-2}}+|z'_j|^{\frac{4}{d-2}}\right)|z_m-z'_m|,\quad k=1,...,l.
			\end{equation}
			In particular, 
			$$
			|f_k(\mathbf{z})|\lesssim \sum_{m=1}^l|z_m|^{\frac{d+2}{d-2}}.
			$$
			\item[(ii)] Let $\mathbf{u}$ and $\mathbf{v}$ be complex-valued functions defined on $\R^d$. Then,
			\begin{equation}\label{gradine}
				\left|\nabla[f_k(\mathbf{u})-f_k(\mathbf{v})]\right|\lesssim \sum_{m=1}^l\sum_{j=1}^l|u_j|^{\frac{4}{d-2}}|\nabla u_m-\nabla v_m|+\sum_{m=1}^l\sum_{j=1}^l\left(|u_j|^{\frac{6-d}{d-2}}+|v_j|^{\frac{6-d}{d-2}}\right)|u_j-v_j||\nabla v_m|.
			\end{equation}
			\item[(iii)] Let $1<p,q,r<\infty$ be such that $\displaystyle\frac{1}{r}=\frac{1}{p}\left(\frac{4}{d-2}\right)+\frac{1}{q}$. Then, for $k=1,\ldots,l$,
			$$
			\Vert \nabla [f_k(\mathbf{u})-f_k(\mathbf{v})]\Vert_{\mathbf{L}^r}\lesssim \Vert\mathbf{u}\Vert_{\mathbf{L}^p}^{\frac{4}{d-2}}\Vert\nabla (\mathbf{u}-\mathbf{v})\Vert_{\mathbf{L}^q}+\left(\Vert \ub\Vert_{L^p}^{\frac{6-d}{d-2}}+\Vert \vb\Vert_{L^p}^{\frac{6-d}{d-2}}\right)\Vert \mathbf{u}-\mathbf{v}\Vert_{\mathbf{L}^p}\Vert\nabla\mathbf{v}\Vert_{\mathbf{L}^q}.
			$$
			In particular, 
			$$
			\Vert \nabla f_k(\mathbf{u})\Vert_{\mathbf{L}^r}\lesssim \Vert\mathbf{u}\Vert_{\mathbf{L}^p}^{\frac{4}{d-2}}\Vert\nabla \mathbf{u}\Vert_{\mathbf{L}^q}.
			$$
					\end{itemize}
	\end{lemma}
	\begin{proof}
		 The Fundamental Theorem of Calculus yields part (i). Part (ii) is established using the chain rule and \cite[Remark 2.3]{cazetal}, observing that for $3 \leq d \leq 5$, we have $\frac{4}{d-2} > 1$. Part (iii) follows from combining part (ii) with Hölder's inequality.
	\end{proof}
		\begin{lemma}\label{lemma22}
		Assume that \ref{H1}, \ref{H2}, \ref{H3} and \ref{H5} hold. 
		\begin{itemize}
			\item[(i)] Let $\mathbf{z},\mathbf{z}'\in\C^l$. Then
			\begin{equation}\label{fk22}
				|\RE F(\zb)-\RE F(\mathbf{z}')|\lesssim \sum_{m=1}^l\sum_{j=1}^l\left(|z_j|^{\frac{d+2}{d-2}}+|z'_j|^{\frac{d+2}{d-2}}\right)|z_m-z_m'|.
			\end{equation}
			In particular, 
			\begin{equation}\label{fk23}
				|\RE F(\mathbf{z})|\lesssim \sum_{j=1}^l|z_j|^{\frac{2d}{d-2}}.
			\end{equation}
			\item[(ii)] Let $\mathbf{u}$ be a complex-valued function defined on $\R^d$. Then
			\begin{equation}\label{fk24}
				\RE \sum_{k=1}^lf_k(\ub)\widebar{u}_k=\RE[2d/(d-2) F(\ub)].
			\end{equation}
		\item[(iii)] Let $\ub:\R^d\rightarrow \C$. Then,
		\begin{equation}\label{regradF}
			\RE\sumk f_k(\ub)\nabla\bar{u}_k=\RE[\nabla F(\ub)].
		\end{equation}
		
			\item[(iv)] We have
			\begin{equation}\label{fk25}
				f_k(x)=\frac{\partial F}{\partial x_k}(x),\quad x\in\R^l.
			\end{equation}
			In addition, $F$ is positive on the positive cone of $\R^l$.
		\end{itemize}
	\end{lemma}
	\begin{proof}
		To prove part (i), we first observe that
		\begin{equation}\label{refk}
			f_k(\zb)=2\frac{\partial}{\partial \widebar{z}_k}\RE F(\zb).
		\end{equation}
		Combining this with \eqref{fk21} and applying the Fundamental Theorem of Calculus, we obtain \eqref{fk22}. For part (ii), we differentiate both sides of \ref{H5} with respect to $\lambda$, evaluate at $\lambda = 1$, and then take the real part, using \ref{H3} to conclude. For part (iii), we first differentiate $F$ with respect to $x_j$ and apply the chain rule. Taking the real part of the result and invoking \ref{H3}, we obtain the desired conclusion. Finally, part (iv) follows directly from \eqref{refk} and \eqref{fk25}.
	\end{proof}
		The next lemma shows that our system is gauge invariant.
	\begin{lemma}\label{GC}
		Assume that \ref{H3} and \ref{H4} hold. For any $\theta\in\R$ and $\zb\in\C^l$, we have
		\begin{itemize}
			\item[(i)] $\RE F\displaystyle\left(e^{i\frac{\sigma_1}{2}\theta}z_1,...,e^{i\frac{\sigma_l}{2}\theta}z_l\right)=\RE F(\zb).$\\
			\item[(ii)] The nonlinearities $f_k$, $k=1,...,l$ satisfy the Gauge condition
			$$
			f_k\left( e^{i\frac{\sigma_1}{2}\theta}z_1,...,e^{i\frac{\sigma_l}{2}\theta}z_l\right)= e^{i\frac{\sigma_k}{2}\theta}f_k(\zb).
			$$
		\end{itemize}
	\end{lemma}
	\begin{proof}
		Part (i) follows from \eqref{refk} and the chain rule. Part (ii) is a direct consequence of part (i).
	\end{proof}
		Another immediate consequence of assumptions \ref{H3} and \ref{H5} is that the nonlinearities  $f_k$ are homogeneous of degree $\frac{d+2}{d-2}$. Indeed, for any $\zb\in\C^l$ and $\lambda>0$, we have
		\begin{equation}\label{25}
			f_k(\lambda\zb)=\lambda^{\frac{d+2}{d-2}}f_k(\zb),\quad k=1,...,l.
		\end{equation}
	
	\subsection{Classical results}
	To conclude this section, we present several well-known results that are useful for our analysis. We begin with the following definition.
	\begin{definition}[Admissible pair]\label{Adpair} A pair $(q,r)$ is admissible if it satisfies
		\begin{equation}\label{admissible} \displaystyle\frac{2}{q}+\frac{d}{r}=\frac{d}{2},
		\end{equation}
		and $2\leq r\leq \frac{2d}{d-2}$.
	\end{definition}
	\begin{lemma}\label{strichartz}[Strichartz estimates \cite[Theorem 2.3.3]{cazenave}] The following inequalities hold.
		\begin{itemize}
			\item[(i)] If $(q,r)$ is an admissible pair, then for all $f\in L^2(\mathbb{R}^d)$
			$$
			\Vert e^{it\Delta}f\Vert_{L^q_tL^r_x(\mathbb{R}\times\mathbb{R}^d)}\lesssim \Vert f\Vert_{L^2_x(\mathbb{R}^d)}.
			$$
			
			\item[(ii)]Let $I$ be a time interval and $t_0\in \bar{I}$. If $(q_1,r_1)$ and $(q_2,r_2)$ are two admissible pairs, then
			$$
			\left\Vert \int_{t_0}^t e^{i(t-s)\Delta}f(\cdot,s)ds\right\Vert_{L^{q_1}_tL^{r_1}_x(I\times\mathbb{R}^d)}\lesssim \Vert f\Vert_{L^{q'_2}_tL^{r'_2}_x(I\times\mathbb{R}^d)}.
			$$
			
		\end{itemize} 
	\end{lemma}
	
	\begin{remark}\label{remstri}
		The above results remain valid if we replace $e^{it\Delta}$ with $U_k(t)$, $k=1,\ldots,l$ (see Definition \ref{Solution}).
	\end{remark}
	\begin{lemma}\label{sobolev}[Sobolev embedding \cite[(A.11)]{tao3}]. Let $1<p<q<\infty$ and $s>0$ satisfying $\frac{1}{p}=\frac{1}{q}+\frac{s}{d}$, then
		$$
		\Vert f\Vert_{L^q(\R^d)}\lesssim \Vert f\Vert_{\dot{\mathbf{H}}^{s,p}(\R^d)}.
		$$
	\end{lemma}
			
\begin{lemma}\label{lema comp}[\cite[Corollary 3.2]{pastor2}]
	Let $I\subset\mathbb{R}$ be an open interval with $0\in I$, $a\in\R$, $b>0$ and $q>1$. Define $\gamma=(bq)^{-1/(q-1)}$ and $f(r)=a-r+br^q$, for $r>0$. Let $G(t)$ be a nonnegative continuous function such that $f\circ G\geq 0$ in $I$. Assume that $a<(1-\delta)\displaystyle\left(1-\displaystyle\frac{1}{q}\displaystyle\right)\gamma$, for some $\delta>0$ sufficiently small, then we have
	\begin{itemize}
		\item[(i)] If $G(0)<\gamma$ then there exists $\delta_1=\delta_1(\delta)>0$ such that $G(t)<(1-\delta_1)\gamma$, for all $t\in I$;
		\item[(ii)] If $G(0)>\gamma$ then there exists $\delta_2=\delta_2(\delta)$ such that $G(t)>(1+\delta_2)\gamma$, for all $t\in I$.
	\end{itemize}
\end{lemma}
	\begin{remark}
		The above result also holds with $\delta=0=\delta_1=\delta_2$. For details see \cite[Lemma 3.1]{pastor2}.
	\end{remark}

\begin{proposition}\label{Virial}[Virial identities]
	Assume that $\mathbf{u}_0\in\mathbf{H}_x^1(\mathbb{R}^d)$ and let $\mathbf{u}$ be the corresponding solution of \eqref{SISTQ}. Let $\varphi\in C_0^{\infty}(\mathbb{R}^d)$ and define
	$$
	V(t)=\int\varphi(x)\left(\sum_{k=1}^l\frac{\alpha_k^2}{\gamma_k}|u_k|^2\right)dx.
	$$
	Then,
	$$
	V'(t)=2\sum_{k=1}^l\alpha_k\mathrm{Im}\int\nabla\varphi\cdot\nabla u_k\widebar{u_k}dx,
	$$
	and
	\begin{equation}\label{virial}
		\begin{split}
			V''(t)&=4\sum_{1\leq m,j\leq 6}\mathrm{Re}\int\frac{\partial^2\varphi}{\partial x_m\partial x_j}\left[\sum_{k=1}^l\gamma_k\partial_{x_j}\bar{u}_k\partial_{x_m}u_k\right]dx\\
			&\quad -\int \Delta^2\varphi\left(\sum_{k=1}^l\gamma_k|u_k|^2\right)dx-\frac{8}{d-2}\mathrm{Re}\int\Delta\varphi F(\mathbf{u})dx.
		\end{split}
	\end{equation}
\end{proposition}
\begin{proof}
	   The proof is an adapted version of the one presented in \cite[Lemma 2.9]{kavian}. Here, Lemma \ref{lemma22} is the key tool to get the desired identities.
\end{proof}	
	
	\section{Existence of ground state solution}\label{secgs}
This section is devoted to prove the existence of ground state solutions as stated in Theorem \ref{ESGS}. We begin by noting that we do not assume the potential function $F$ is strictly positive away from the origin. The following result establishes a property of the potential energy $P(\psib)$ that allows us to address this issue by placing our problem in an appropriate set, rather than $\Dz$.
	\begin{lemma}\label{lema31}
		Let $\mathcal{N}:=\{\boldsymbol{\psi}\in \dot{\mathbf{H}}^1(\mathbb{R}^d);\,\, P(\boldsymbol{\psi})>0\}$. If $\mathcal{C}$ denotes the set of all non-trivial solutions of  \eqref{SIST3}, then  $\mathcal{C}\subset \mathcal{N}$. 
	\end{lemma}
	\begin{proof}
		Assume $\boldsymbol{\psi}\in\mathcal{C}$. By inserting $\mathbf{g}=\boldsymbol{\psi}$ in \eqref{solfraca} we obtain
		$$
		\int |\nabla \psi_k|^2dx=\int f_k(\boldsymbol{\psi})\psi_kdx,\quad k=1,...,l.	
		$$
		Using Lemma \ref{lemma22} (ii) we get 
		\begin{equation}\label{36}
			K(\psib)=\left(\frac{2d}{d-2}\right)P(\psib).
		\end{equation} 
		Since $\psib$ is non-trivial, it follows that $P(\psib)>0$ and $\psib\in\mathcal{N}$ as desired.
	\end{proof}
	
	Now, let us introduce the functionals 
	\begin{equation}\label{J}
		J(\psib):=\frac{K(\psib)^{\frac{d}{d-2}}}{P(\psib)},\quad \psib\in\mathcal{N}.
	\end{equation}
	and
	\begin{equation}\label{E}
		\mathcal{E}(\psib):=K(\psib)-2P(\psib).
	\end{equation}
	
	\begin{remark}\label{obs32}
		Using \eqref{36}, we observe that $S(\psib) = \frac{2}{d-2}P(\psib)=\frac{1}{2}\mathcal{E}(\psib)$. Thus, by the definition of $J$, we obtain
$$
J(\psib)=\frac{K(\psib)^{\frac{d}{d-2}}}{P(\psib)}=\left(\frac{2d}{d-2}\right)^{\frac{d}{d-2}}P(\psib)^{\frac{2}{d-2}}=\frac{2}{d-2}d^{\frac{d}{d-2}}S(\psib)^{\frac{2}{d-2}}.
$$ 
Hence, a non-trivial solution to \eqref{SIST3} is a ground state if and only if it has the least energy $\mathcal{E}$ among all solutions, which is equivalent to minimizing $J$.
	\end{remark}

	\subsection{Critical and localized Sobolev-type inequalities.}
	
	For the remainder of this section, we assume that the vector $\ub$ is a real-valued function. First, applying Lemma \ref{lemma22} and Sobolev's inequality, we obtain
	\begin{equation*}
		|P(\ub)|\leq \int |\RE F(\ub)|dx\leq C\sum_{k=1}^l \Vert u_k\Vert_{L^{\frac{2d}{d-2}}}^{\frac{2d}{d-2}}\leq C_0 K(\ub)^{\frac{d}{d-2}}
	\end{equation*}
	where $C_0$ is a positive constant depending on $\gamma_k$. Hence, we obtain the following critical Sobolev-type inequality
	\begin{equation}\label{SCI}
		P(\ub)\leq C K(\ub)^{\frac{d}{d-2}},
	\end{equation}
	for some positive constant $C$. In particular, there exists a constant $C > 0$ such that
	$$
	\frac{1}{C}\leq J(\ub), \quad \ub\in\mathcal{N}
	$$
and we define the optimal constant by
	\begin{equation}\label{Min}
		C_{\rm opt}^{-1}:=\inf\{J(\ub);\,\ub\in\mathcal{N}\}.
	\end{equation} 

As stated in Remark \ref{obs32}, a ground state solution exists once the above infimum is achieved. Let $({\ub_m}) \subset \mathcal{N}$ be a minimizing sequence for \eqref{Min}, that i $J(\ub_m) \to C_{\rm opt}^{-1}$. A straightforward calculation shows that $K(\big\bracevert\! \mathbf{u}\!\big\bracevert)\leq K(\ub)$ , and combined with \ref{H6}, it follows that $J(\big\bracevert\! \mathbf{u}\!\big\bracevert)\leq J(\ub)$. Thus, we can assume, without loss of generality, that minimizing sequences are non-negative. Moreover, by the homogeneity of the functionals $K$ and $P$, it also follows that $C_{opt}=I^{-\frac{d}{d-2}}$, where $I$ is given by
	\begin{equation}\label{MinNorm}
		I:=\inf\{K(\ub);\, \ub\in\mathcal{N},\,\, P(\ub)=1\}.
	\end{equation}
	Thus, solving the minimization problem \eqref{Min} is equivalent to solving \eqref{MinNorm}. Finally, the functionals $K$ and $P$ are also invariant under the transformation
	\begin{equation}\label{obs34ii}
		\ub\mapsto\ub^{R,y}=R^{-\frac{d-2}{2}}\ub\left( R^{-1}(x-y)\right)
	\end{equation}
where $R>0$, which ensures ensures that our minimization problem is invariant under dilations and translations. 

To conclude this section, we recall a localized version of Sobolev's inequality (see, for instance, \cite[Corollary 9]{FM} and also \cite[Corollary 3.9]{pastor3} for a vector-valued version).
		\begin{lemma}[Localized Sobolev's inequality]\label{corol39}
		Assume $\ub\in\Dz$ satisfies $\ub>0$. For any $\delta>0$ we can find $C(\delta)>0$ such that if $0<r<R$ satisfies $r/R \leq C(\delta)$, then
		\begin{equation}\label{337}
			\int_{B(x,r)}F(\ub)dy\leq I^{-\frac{d}{d-2}}\left[\int_{B(x,R)}\sumk\gamma_k|\nabla u_k|^2dy+\delta K(\ub)\right]^{\frac{d}{d-2}}
		\end{equation}
		and
		\begin{equation}\label{338}
			\int_{\mathbb{R}^d\backslash B(x,R)}F(\ub)dy\leq I^{-\frac{d}{d-2}}\left[\int_{\mathbb{R}^d\backslash B(x,r)}\sumk\gamma_k|\nabla u_k|^2dy(2\delta +\delta^2) K(\ub)\right]^{\frac{d}{d-2}}.
		\end{equation}
	\end{lemma}
	 
	\subsection{Concentration-compactness} 
	As previously noted, the next step in addressing the minimization problem \eqref{MinNorm} is to apply the concentration-compactness method from \cite{lions2}. We begin with a modified version of Lemma 1.7.4 from \cite{cazenave} regarding the concentration function, defined as follows
	\begin{equation}\label{concfunc}
	 Q_m(R):=\sup_{y\in\mathbb{R}^d}\int_{B(y,R)}F(\ub_m)dx,\quad R>0.
	 \end{equation}
	\begin{lemma}\label{lema310}
		Let $(\ub_m)\subset{\mathbf{L}}^{\frac{2d}{d-2}}(\mathbb{R}^d)$ be such that $\ub_m\geq 0$ and $\int F(\ub_m)dx=1$, for any $m\in\mathbb{N}$. Let $Q_m(R)$ be the concentration function of $F(\ub_m)$ defined by \eqref{concfunc}. Then, for each $m\in\N$, there is $y=y(m,R)$ such that
		$$
		Q_m(R)=\int_{B(y,R)}F(\ub_m)dx. 
		$$
	\end{lemma} 
	\begin{proof}
		Fix $m\in\mathbb{N}$. Given $R>0$, there exists a sequence $(y_j)\subset\R^d$ such that 
		$$
		Q_m(R)=\lim_{j\rightarrow\infty}\int_{B(y_j,R)}F(\ub_m)dx>0.
		$$
Thus, one must find $j_0$ such that if $j\geq j_0$ then $\int_{B(y_j,R)}F(\ub_m)dx\geq \epsilon$, for some $\epsilon>0$. We claim that the sequence $(y_j)$ is bounded. Indeed, if it were not, then up to a subsequence, the balls $B(y_j, R)$ and $B(y_i, R)$ would be disjoint for all $i \neq j$. Then
		$$
		1=\int F(\ub_m)dx\geq \sum_{j\geq j_0}\int_{B(y_j,R)}F(\ub_m)dx=\infty,
		$$
which is a contradiction. Therefore, the sequence $(y_j)$ admits a convergent subsequence with limit $y = y(m, R)$ and the proof is completed by applying the dominated convergence theorem.
	\end{proof}

We now introduce notation for the concentration-compactness framework.  Let $X$ be a locally compact Hausdorff space. We denote by $\mathcal{C}_b(X)$ the space of all bounded continuous functions on $X$, and by $\mathcal{C}_c(X)$ the space of all continuous functions with compact support on $X$. Moreover, $\mathcal{M}_+(X)$ denotes the Banach space of all non-negative measures on $X$, $\mathcal{M}_+^b(X)$ the space of all finite (or bounded) measures and $\mathcal{M}_+^1$ the space of all probability measures. Given two measures $\nu$ and $\mu$, we write $\nu \ll \mu$ to indicate that $\nu$ is absolutely continuous with respect to $\mu$. For any $\mu \in \mathcal{M}_+^b(X)$, the quantity $|\mu| := \mu(X)$ is called the mass of $\mu$. Let us introduce some notions of convergence in the context of measures.
	\begin{definition}
	$\bf{}$
		\begin{itemize}
			\item[(i)] A sequence $(\mu_m)\subset\mathcal{M}_+$ is said to converge vaguely to $\mu$ in $\mathcal{M}_+(X)$, denoted by $\mu_m\overset{\ast}{\rightharpoonup}\mu$, if $\int_Xfd\mu_m\rightarrow\int_Xfd\mu$ for all $f\in\mathcal{C}_c(X)$.
			
			\item[(ii)] A sequence $(\mu_m)\subset\mathcal{M}_+^b(X)$ is said to converge weakly to $\mu$ in $\mathcal{M}_+^b(X)$, denoted by $\mu_m\rightharpoonup\mu$, if $\int_X fd\mu_m\rightarrow\int_X fd\mu$, for all $f\in \mathcal{C}_b(X)$.
			
			\item[(iii)] A sequence $(\mu_m)\subset\mathcal{M}_+^b$ is said to be uniformly tight if, for every $\epsilon>0$, there exists a compact subset $K_\epsilon\subset X$ such that $\mu_m(X\backslash K_\epsilon)\leq \epsilon$ for all $m$. We also say that a set $\mathcal{H}\subset \mathcal{M}_+(X)$ is vaguely bounded if $\sup_{\mu\in\mathcal{H}}\left|\int_Xf d\mu\right|<\infty$ for all $f\in\mathcal{C}_c(X)$.
		\end{itemize}
	\end{definition}
The following result is motivated by Lemma I.1 from \cite{lions1}. For a complete proof we refer the reader to \cite[Lemma 23]{FM}.
	
	\begin{lemma}[Concentration-compactness lemma I] \label{lema35}	
		Suppose that $(\nu_m)$ is a sequence in $\mathcal{M}_+^1(\mathbb{R}^d)$. Then, there is a subsequence, still denoted by $(\nu_m)$, such that one of the following conditions holds:
		\begin{itemize}
			\item[(i)](Vanishing) For all $R>0$ it holds
			$$
			\lim_{m\rightarrow\infty}\left(\sup_{x\in\mathbb{R}^d}\nu_m(B(x,R))\right)=0.
			$$
			
			\item[(ii)](Dichotomy) There is a number $\lambda\in(0,1)$ such that for all $\epsilon>0$ there exists $R>0$ and a sequence $(x_m)$ with the following property: given $R'>R$
			$$
			\nu_m(B(x_m,R))\geq \lambda-\epsilon,
			$$
			$$
			\nu_m(\mathbb{R}^d\backslash B(x_m,R'))\geq 1-\lambda-\epsilon,
			$$
			for $m$ sufficiently large.
			
			\item[(iii)](Compactness) There exists a sequence $(x_m)\subset\mathbb{R}^d$ such that for each $\epsilon>0$ there is a radius $R>0$ with the property
			$$
			\nu_m(B(x_m,R))\geq 1-\epsilon,
			$$
			for all $m$.
		\end{itemize}
	\end{lemma}
To determine a minimizer for \eqref{MinNorm}, the next step is to construct an appropriate sequence of probability Radon measures. By Lemma \ref{lema35}, up to a subsequence, this sequence will satisfy one of the three conditions outlined. We rule out vanishing and dichotomy, ensuring the sequence's compactness. Consequently, we achieve vague convergence in $\mathcal{M}_+^b(\mathbb{R}^4)$. This convergence enables the application of what we call the concentration-compactness lemma II, inspired by the limit case lemma in \cite{lions2}. In essence, this lemma ensures dilation invariance for the minimization problem.
	
	\begin{lemma}[Concentration-compactness lemma II] \label{lema36}
		Let $(\ub_m) \subset\Dz$ be a sequence such that $\ub_m\geq 0$ and 
		\begin{equation*}
			\left\{\begin{array}{lcc}
				\ub_m\rightharpoonup \ub, & \hbox{in}& \Dz,\\
				\mu_m:=\displaystyle\sum_{k=1}^l\gamma_k|\nabla u_{km}|^2dx\overset{\ast}{\rightharpoonup} \mu & \hbox{in}& \mathcal{M}_+^b(\mathbb{R}^d)\\
				\nu_m:=F(\ub_m) dx \overset{\ast}{\rightharpoonup} \nu,& \hbox{in}& \mathcal{M}_+^b(\mathbb{R}^d). \\
					\end{array}\right.
		\end{equation*}
		Then,
		\begin{itemize}
			\item[(i)] There exists an at  most countable set $J$, a family of distinct points $\{x_j\in \mathbb{R}^d;\, j\in J\}$, and a family of non-negative numbers $\{a_j;\, j\in J\}$ such that
			\begin{equation}\label{315}
				\nu=F(\ub)dx+\sum_{j\in J}a_j\delta_{x_j}.
			\end{equation}
			
			\item[(ii)]Moreover, we have
			\begin{equation}\label{316}
				\mu \geq  \sum_{k=1}^l\gamma_k|\nabla u_k|^2dx +\sum_{j\in J}b_j\delta_{x_j}
			\end{equation}
			fore some family $\{b_j;\, j\in J\},b_j>0,$ such that
			\begin{equation}\label{317}
				a_j\leq I^{-\frac{d}{d-2}}b_j^{\frac{d}{d-2}},\quad \mbox{for all}\quad j\in J.
			\end{equation}
			In particular, $\displaystyle\sum_{j\in J}a_j^{\frac{d-2}{d}}<\infty$.
		\end{itemize}
	\end{lemma}
	\begin{proof}
		The proof follows the same approach as in the quadratic case presented in \cite[Lemma 3.6]{pastor3}, with suitable modifications for our context.
	\end{proof}
	
We now apply the preceding two lemmas to a minimizing sequence for \eqref{MinNorm}. Here, we provide a sketch of the proof and for a detailed version  the reader may consult \cite[Theorem 3.11]{pastor3} and \cite[Theorem 3.12]{HP}.
	\begin{theorem}\label{TEO311}
		Suppose that $(\ub_m)$ is a minimizing sequence for \eqref{MinNorm} with $\ub_m\geq 0$. Then, up to translation and dilation $(\ub_m)$ is relatively compact in $\mathcal{N}$, that is, there exist a subsequence $(\ub_{m_j})$ and sequences $(R_j)\subset\mathbb{R}$, $(y_j)\subset \mathbb{R}^d$ such that  $\vb_j$ given by
		$$
		\vb_j:=R_j^{-\frac{d-2}{2}}\ub_{m_j}(R_j^{-1}(x-y_j)),
		$$
strongly converges in $\mathcal{N}$ to some nontrivial $\vb$, which minimizes \eqref{MinNorm}.
	\end{theorem}
	\begin{proof}[Sketch of the proof]
		We first consider $(\ub_m)\subset\mathcal{N}$ to be any minimizing sequence of \eqref{MinNorm} with $\ub_m\geq 0$. Now, notice that by using the invariance \eqref{obs34ii} and Lemma \ref{lema310} we obtain sequences $(R_j)\subset\R$ and $(y_j)\subset\R^d$ such that
		\begin{equation*}
			\vb_j:=R_j^{-\frac{d-2}{2}}\ub_{m_j}(R_j^{-1}(x-y_j)),
		\end{equation*}
		satisfies
		\begin{equation}\label{341}
			\sup_{y\in\mathbb{R}^d}\int_{B(y,1)}\varphi(\vb_m)dx=\int_{B(0,1)}\varphi(\vb_m)dx=\frac{1}{2}.
		\end{equation}
		Applying the invariance property \eqref{obs34ii} again, we conclude that $(\vb_m)$ is also a minimizing sequence for \eqref{MinNorm}. In particular, $(\vb_m)$ is uniformly bounded in $\mathcal{N}$. Hence, up to a subsequence, there exists $\vb \in \Dz$ such that
		\begin{equation}\label{345}
			\vb_m\rightharpoonup \vb\quad \mbox{in }  \Dz.
		\end{equation} 
		We now define the sequences of measures $(\mu_m)$ and $(\nu_m)$ by
		$$
		\mu_m:=\sum_{k=1}^l\gamma_k|\nabla v_{km}|^2dx,\quad\hbox{and}\quad \nu_m:=F(\vb_m) dx.
		$$
Since $(\vb_m)$ is a minimizing sequence for \eqref{MinNorm}, the sequence of measures $(\nu_m)$ is a sequence of probability measures. By Lemma \ref{lema35}, exactly one of the following three cases must occur: vanishing, dichotomy, or compactness. It follows immediately from \eqref{341} that vanishing cannot occur. To rule out the dichotomy case, we proceed by contradiction and apply Corollary \ref{corol39}. Therefore, the compactness case must hold. This means there exists a sequence $(x_m) \subset \mathbb{R}^d$ such that for every $\epsilon > 0$, there exists $R > 0$ with
		\begin{equation}\label{350}
			\nu_m(B(x_m,R))\geq 1-\epsilon,\quad \mbox{for all} \quad m.
		\end{equation}
As a consequence of this compactness property combined with \eqref{341}, we deduce that the sequence $(\nu_m)$ is uniformly tight. By Theorems 31.2 and 30.6 in \cite{Bauer}, it admits a subsequence that converges weakly to some measure $\nu\in\mathcal{M}_+^b(\R^d)$, that is
		\begin{equation}\label{351}
			\int fd\nu_m\rightarrow\int fd\nu,\quad \mbox{for all} \quad f\in{C}_b(\mathbb{R}^4).
		\end{equation}
In particular, by taking $f\equiv 1$, we observe that
		\begin{equation}\label{352}
			\nu(\R^d)=\lim_{m\rightarrow\infty}\nu_m(\R^d)=1
		\end{equation}
and hence $\nu\in\mathcal{M}_+^1(\R^d)$.
		
		Next, the uniformly boundness of $(K(\vb_m))$ implies that $(\mu_m)$ is also vaguely bounded. Then, up to a subsequence, there exists $\mu\in\mathcal{M}_+^b(\R^d)$ such that 
		\begin{equation}\label{353}
			\mu_m \overset{\ast}{\rightharpoonup}\mu\quad\mbox{in}\quad \mathcal{M}_+^b(\mathbb{R}^d).
		\end{equation}
		Therefore, combining \eqref{345}, \eqref{351} and \eqref{353}, we apply Lemma \ref{lema36} to deduce
		\begin{equation}\label{354}
			\mu\geq \sum_{k=1}^l\gamma_k|\nabla v_k|^2dx+\sum_{j\in J}a_j \delta_{x_j}\quad \hbox{and}\quad\nu=F(\vb)dx+\sum_{j\in J}b_j \delta_{x_j},
		\end{equation}
for some family $\{x_j\in\R^d;\,\,j\in J\}$, where $J$ is countable, and non-negative numbers $a_j,b_j$ satisfying
		\begin{equation}\label{355}
			a_j\leq I^{-\frac{d}{d-2}}b_j^{\frac{d}{d-2}},\quad \mbox{for all} \quad j\in J
		\end{equation}
		with the series $\sum_{j\in J}a_j^{\frac{d-2}{d}}$ convergent. Hence, \eqref{SCI}, \eqref{352} and \eqref{355} lead to
		\begin{equation}\label{356}
			\begin{split}
				I&=\liminf_{m\rightarrow\infty}\mu_m(\mathbb{R}^d)\geq\mu(\mathbb{R}^d)
				\geq K(\vb)+\sum_{j\in J}b_j
				\geq I\left[P(\vb)^{\frac{d-2}{d}}+\sum_{j\in J}a_j^{\frac{d-2}{d}}\right]\\
				&\geq I\left[P(\vb)+\sum_{j\in J}a_j\right]^{\frac{d-2}{d}}
				=I[\nu(\mathbb{R}^d)]^{\frac{d-2}{d}}
				=I,
			\end{split}
		\end{equation}
where we have used that $\lambda\mapsto \lambda^{\frac{d-2}{d}}$ is a strictly concave function. Notice that $a_j$ must vanishes for all $j\in J$. Indeed, if there exists some $a_{j_0}\neq 0$, then by \eqref{352} and \eqref{354} we have $\nu=a_{j_0}\delta_{x_{j0}}$ which implies
		\begin{equation}\label{357}
			1=\nu(\mathbb{R}^4)=a_{j_0}.
		\end{equation}
		However, the normalization \eqref{341} leads to
		$$
		\frac{1}{2}\geq\lim_{m\rightarrow\infty}\nu_m(B(x_{j_0},1))=\nu(B(x_{j_0},1))=\int_{B(x_{j_0},1)}d\nu=a_{j_0},
		$$
which contradicts the previous equality.Therefore, it must be the case that $\nu=F(\vb)dx$ which, by \eqref{352} implies $\vb\in\mathcal{N}$. To complete the proof, note that by the definition of $I$, the fact that $\vb\in\mathcal{N}$ and the lower semi-continuity of the weak convergence in \eqref{345}, we deduce that $K(\vb_m)=I$ and that $\vb_m\rightarrow\vb$ strongly in $\mathcal{N}$. 
	\end{proof}
	We are now in a position to prove Theorem \ref{ESGS}. 
	
	\begin{proof}[Proof of Theorem \ref{ESGS}]
		We begin by applying Theorem \ref{TEO311} to obtain a nontrivial minimizer $\vb$ of  \eqref{MinNorm}. By Lagrange’s multiplier theorem, there exists a constant $\lambda$ such that for any $\mathbf{g}\in\Dz$
		\begin{equation*}
			2\gamma_k \int \nabla v_k\cdot\nabla g_kdx= \lambda\int f_k(\vb)g_kdx,\quad k=1,...,l.
		\end{equation*}
Taking $\mathbf{g}=\vb$ and applying Lemma \ref{lemma22}-$(ii)$ we deduce that  $\lambda> 0$. Next, we define 
$$
\psib_0(x):=\left(\frac{\lambda}{2}\right)^{\frac{4}{d-2}}\vb(x)
$$ 
and verify that $\psib_0$ is indeed a ground state solution of \eqref{SIST3}. For each $k=1,...,l$, 
		\begin{equation*}
			\gamma_k\int\nabla \psi_{k0}\cdot\nabla g_kdx=\left(\frac{\lambda}{2}\right)^{\frac{4}{d-2}}\gamma_k\int\nabla v_k\cdot\nabla g_k=\int\left(\frac{\lambda}{2}\right)^{\frac{d+2}{d-2}} f_k(\vb)g_kdx=\int f_k(\psib_0)g_kdx,
		\end{equation*}
which shows $\psib_0$ solves \eqref{SIST3}. By the homogeneity of the functionals $P$ and $K$, it follows that $J(\psib_0)=J(\vb)$. Since $\vb$ minimizes \eqref{MinNorm}, it is also a minimizer of \eqref{Min}, and therefore, $\psib_0$ is a minimizer of \eqref{Min}. Applying Remark \ref{obs32}, we conclude that $\psib_0$ is a nontrivial ground state solution of \eqref{SIST3}.
	\end{proof}
	
As a corollary, we obtain the sharp constant for the inequality \eqref{SCI} restricted to the set $\mathcal{N}$.
	
	\begin{corollary}\label{corol314} The following inequality holds
		\begin{equation}\label{360}
			P(\ub)\leq C_{\rm opt}K(\ub)^{\frac{d}{d-2}},\quad \mbox{for all} \quad \ub\in\mathcal{N}
		\end{equation}
with the optimal constant given by
		\begin{equation}\label{optimal}
			C_{\rm opt}=\frac{1}{C_d}\left(\frac{2}{\mathcal{E}(\psib)}\right)^{\frac{2}{d-2}},
		\end{equation}
where $C_d=\frac{2}{d-2}d^{\frac{d}{d-2}}$ and $\psib$ is any ground state solution of \eqref{SIST3}.
	\end{corollary}
	\begin{remark}
The uniqueness of ground state solutions for \eqref{SIST3}, as far as we know, remains an open problem. However, since all ground states of \eqref{SIST3} share the same energy $\mathcal{E}$, the constant $C_{\rm opt}$ is independent of the particular ground state chosen.
	\end{remark}

	\section{Global well-posedness and scattering in $\Dz$}\label{secgwp}

		\subsection{Local theory}\label{seclocal}
		
The proof of Theorem \ref{MAIN} can be derived by adapting the strategy in \cite[Theorem 2.5]{KM} (see also \cite[Chapter 3]{killip3}), with straightforward adjustments for the vectorial setting, and is omitted here. Indeed, we can establish the following slightly more general result.
	\begin{proposition}\label{BCL}
	Let $3\leq d\leq 5$. Assume that \ref{H1} and \ref{H2} hold and let $t_0\in I$ an interval. For any $\mathbf{u}_0\in \dot{\mathbf{H}}^1(\mathbb{R}^d)$ with $\|\mathbf{u}_0\|_{\dot{\mathbf{H}}^1}\leq A$, there exists $\delta=\delta(A)>0$ such that, if 
$$
\Vert U_k(t-t_0)u_k(t_0)\Vert_{S(I)}\leq\delta, \quad \kl,
$$
there exists a unique corresponding solution $\mathbf{u}$ of \eqref{SISTR} in $I\times\R^d$ with $\mathbf{u}\in C(I;\dot{\mathbf{H}}^1(\R^d))$ and
	$$
	\Vert \nabla \ub\Vert_{\Xb(I)}<\infty,\quad \Vert \ub\Vert_{\Sb(I)}\leq 2\delta.
	$$
	\end{proposition}
An immediate consequence, in view of the Sobolev embedding and Strichartz estimates, is the global existence for sufficiently small initial data.
	\begin{corollary}[Small data theory]\label{smalldata} There exits $\delta_0>0$ such that if $\Vert \nabla \mathbf{u}_0\Vert_{\mathbf{L}^2}\leq \delta_0$  then the solution $\mathbf{u}$ to the system \eqref{SISTR} obtained in Proposition \ref{BCL} can be extended to any interval $I$.
	\end{corollary}
	
We also say that a solution $\mathbf{u}$ of \eqref{SISTR} given by Proposition \ref{BCL} blows-up forward in time if 
			$$
			\Vert \mathbf{u}\Vert_{\Sb([t_0, \sup I))}=\infty,
			$$
			and $\mathbf{u}$ blows-up backward in time, if 
			$$
			\Vert \mathbf{u}\Vert_{\Sb((\inf I, t_0])}=\infty.	
			$$
We say that $\mathbf{u}$ blows-up in finite time, if it blows-up both forward and backward in time. The following criterion is based on \cite[Lemma 2.11]{KM}.
	\begin{proposition}\label{standblow}[Standard finite blow-up criterion] Let $\mathbf{u}$ be the maximal-lifespan solution defined on interval $I=(-T^-, T^+)$ with $t_0\in I$ and initial data $\mathbf{u}(t_0)=\mathbf{u}_0$. If $T^+<\infty$ ($T^-<\infty$), then $\mathbf{u}$ blows-up forward (backwards) in time. 
	\end{proposition}
	
	\begin{definition}[Non-linear profile]\label{non-linear profile}Let $\vb_0\in\Dz$, $\vb=\mathbf{U}(t)\vb_0$ and $(t_n)$ to be a sequence with $\lim t_n=\tilde{t}\in[-\infty,\infty]$. We say that $\ub(t,x)$ is a non-linear profile associated to $(\vb_0,(t_n))$ if there exists an interval $I$ containing $\tilde{t}$ such that $\ub$ is a solution of \eqref{SISTR} in $I$ and 
		$$
		\lim_{n\rightarrow\infty}\Vert \ub(t_n,\cdot)-\vb(t_n,\cdot)\Vert_{\Dz}=0.
		$$
		
	\end{definition}
	There is always a non-linear profile associated to $(\vb_0,(t_n))$. Also, if $\ub^1$ and $\ub^2$ are both non-linear profiles associated to $(\vb_0,(t_n))$ in an interval $I$, then $\ub^1\equiv \ub^2$ on $I$. For details, see \cite[Remark 2.13]{KM}.

	The following lemma establishes a stability result. Here, we present a modified version of the result from \cite{KM}, which is sufficient for our needs. A more general result for the NLS equation can be found in \cite{tao}.
	\begin{lemma}\label{stability}[Long time perturbation]
		Let $I$ be a compact interval containing $t_0$ and $\mathbf{v}:I\times \mathbb{R}^6\rightarrow\mathbb{C}^l$ be an approximate solution of \eqref{SISTR} in the sense that
		$$
		i\alpha_k\partial_t v_k+\gamma_k\Delta v_k + f_k(\mathbf{v}) =e_k, \quad  k=1,\ldots,l, \quad (t,x)\in I\times\R^d,
		$$
		for some function $\mathbf{e}=(e_1,...,e_l)$. Assume that
		\begin{equation}\label{v54}
			\Vert \mathbf{v} \Vert_{\mathbf{L}^\infty_t \dot{\mathbf{H}}_x^1}\leq A,
		\end{equation}
		\begin{equation}\label{v55}
			\Vert\mathbf{v}\Vert_{\mathbf{S}(I)}\leq M,
		\end{equation}
		where $A, M$ are positive constants. Let $\mathbf{u}_0\in\dot{\mathbf{H}}^1(\mathbb{R}^d)$. Assume also that
		\begin{equation}\label{v56}
			\Vert \mathbf{u}_0-\mathbf{v}(t_0)\Vert_{\dot{\mathbf{H}}_x^1}\leq A',
		\end{equation}
		\begin{equation}\label{v57}
			\Vert \nabla \mathbf{e} \Vert_{\mathbf{L}_t^2 \mathbf{L}_x^{\frac{2d}{d+2}}}\leq \epsilon,
		\end{equation}
		\begin{equation}\label{evolution}
			\Vert	\mathbf{U}(t-t_0)[\mathbf{u}_0-\mathbf{v}(t_0)]\Vert_{\mathbf{S}(I)}\leq\epsilon.
		\end{equation}
		Then, there exists $\epsilon_0=\epsilon_0(M,A,A',d)$ and a unique solution $\mathbf{u}:I\times\mathbb{R}^d\rightarrow\mathbb{C}^l$ to \eqref{SISTR} with initial data $\mathbf{u}(t_0,x)=\mathbf{u}_0$ such that for $0<\epsilon<\epsilon_0$, we have
		\begin{equation}\label{v100}
			\Vert \mathbf{u}\Vert_{\mathbf{S}(I)}\leq C(M,A,A',d)
		\end{equation}
		\begin{equation}\label{v101}
			\Vert  \mathbf{u}(t)-\mathbf{v}(t)\Vert_{\dot{\mathbf{H}}_x^1}\leq C(A,A',M,d) (A' +\epsilon), \quad \mbox{for all} \quad t\in I.
		\end{equation}
	\end{lemma}
	\begin{proof}
We follow the proof of \cite[Theorem 5.3]{KTV14}. First note that assumption \eqref{v55} implies that $\Vert \nabla \mathbf{v} \Vert_{\Xb(I)}$ is finite. Indeed, for $\zeta=\zeta(d)$, \eqref{v55} allow us to split the interval $I$ into $N_0=N_0(M,\zeta)$ intervals $J_j$ so that $\Vert \mathbf{v}\Vert_{\mathbf{S}(J_j)}\leq\zeta$. Therefore, using Strichartz's estimates, Lemma \ref{consequences} and \eqref{v54}, we have, for $k=1,\ldots,l$, 
	\begin{equation}\label{vwj}
		\begin{split}
			\Vert \nabla v_k \Vert_{{W}(J_j)} &\lesssim \Vert v_k\Vert_{L_t^\infty \dot{H}_x^1(I\times\mathbb{R}^d)}+\Vert \nabla f_k(\mathbf{v})\Vert_{L_t^2 L_x^{\frac{2d}{d+2}}(J_j\times\mathbb{R}^6)}+\Vert \nabla e_k\Vert_{L_t^2 L_x^{\frac{2d}{d+2}}(I\times\mathbb{R}^d)}\\
			&\lesssim A+ \Vert \mathbf{v} \Vert_{\mathbf{S}(J_j)}^{\frac{4}{d-2}}\Vert \nabla\mathbf{v} \Vert_{\Xb(J_j)} +\epsilon\\
			&\lesssim A+\zeta^{\frac{4}{d-2}}\Vert \nabla\mathbf{v} \Vert_{\Xb(J_j)}+\epsilon. 
		\end{split}
	\end{equation}
Thus, taking $\zeta$ small enough, and summing this bound over all the intervals $J_j$, we obtain $\Vert \nabla \mathbf{v} \Vert_{\Xb(I)}\leq C(A,M)$. The previous bound now implies that for $\eta=\eta(d)$ small to be chosen later, we can divide the interval $I$ into $N_1=N_1(A,M,\eta)$ intervals $I_j$ so that
\begin{equation}\label{vdelta}
		\Vert \nabla \mathbf{v} \Vert_{\Xb(I_j)}\leq\eta.	
	\end{equation}

We may assume, without loss of generality, that $t_0 = \inf I$. Next write $\mathbf{w}=\mathbf{u} - \mathbf{v}$. By the well-posedness theory, the solution $\mathbf{u}$ is already known to exist. To establish the bounds \eqref{v100}-\eqref{v101}, we consider the integral formulation
	\begin{equation}\label{wint}
		w_k(t)=U_k(t-t_0) w_k(t_0)+i\int_{t_0}^tU_k(t-s)[f_k(\mathbf{u})-f_k(\mathbf{v})+e_k]ds,\quad k=1,\ldots,l.
	\end{equation}
Set
$$
B_k(t):=\Vert \nabla[f_k(\ub)-f_k(\vb)]\Vert_{L_t^{2}L_x^{\frac{2d}{d+2}}([t_0,t]\times\R^d)}\quad \mbox{and}\quad B(t)=\sumk B_k(t).
$$	
Let $I_n=[a_n,a_{n+1}]$, with $a_1=t_0$. Using the integral equation \eqref{wint}, Sobolev embbeding, Strichartz's inequality and \eqref{v57} we deduce
		\begin{equation}\label{wk1}
			\begin{split}
				\Vert w_k \Vert_{S([t_0,t])}&\lesssim \Vert	U_k(t-t_0) w_k(t_0)\Vert_{{S}([t_0,t])}+B_k(t) + \Vert \nabla
				e_k \Vert_{L_{t}^2L_x^{\frac{2d}{d+2}}(I\times\mathbb{R}^d)}
				\\
				&\lesssim B_k(t)+\epsilon
			\end{split}
		\end{equation}
and
\begin{equation}\label{wk2}
			\begin{split}
				\Vert \nabla w_k \Vert_{X([t_0,t])}&\lesssim \Vert w_k(t_0)\Vert_{\dot{H}^1_x}+B_k(t) + \Vert \nabla
				e_k \Vert_{L_{t}^2L_x^{\frac{2d}{d+2}}(I\times\mathbb{R}^d)}
				\\
				&\lesssim A'+ B_k(t)+\epsilon.
			\end{split}
		\end{equation}
Using Lemma \ref{consequences}, \eqref{vdelta} and the last two inequalities, we get
		\begin{equation}\label{tese414}
		\begin{split}
			B_k(t)&\lesssim \Vert \nabla v_k \Vert_{X([t_0,t])}\Vert w_k \Vert_{S([t_0,t])} (\Vert w_k \Vert_{S([t_0,t])}+\Vert v_k \Vert_{S([t_0,t])})^{\frac{6-d}{d-2}}\\
			&\quad +\Vert \nabla w_k \Vert_{X([t_0,t])}(\Vert w_k \Vert_{S([t_0,t])}+\Vert v_k \Vert_{S([t_0,t])})^{\frac{4}{d-2}}\\
			&\lesssim \eta[B(t)+\epsilon][B(t)+\epsilon+\eta]^{\frac{6-d}{d-2}}+[B(t)+\epsilon+A'][B(t)+\epsilon+\eta]^{\frac{4}{d-2}}.
			\end{split}
		\end{equation}
		Summing over $k$ and taking $\epsilon$ and $\eta$ small enough, depending  on the dimension $d$ and $2cA'$ for some $c>1$ related to the Strichartz estimates from Lemma \ref{strichartz} (it is necessary to increase $A'$ since the $\dot{H}^1$ norm will grow over time in the interaction argument), a continuity argument implies
		\begin{equation}\label{v513}
			B(t_1)\leq C_1\epsilon.
		\end{equation}
		
From \eqref{wk1} and \eqref{wk2}, we also obtain
$$
\Vert \mathbf{w} \Vert_{\mathbf{S}(I_1)}\leq C_1  \epsilon \quad \mbox{and} \quad \Vert \nabla \wb \Vert_{\Xb(I_1)}\leq  cA'+C_1 \epsilon. 
$$
		
Next, we obtain similar estimates on the interval $I_2 = [a_2, a_3]$.  To this end, we apply the integral equation \eqref{wint} at $t = a_2$, and use the identity $\mathbf{U}(t - a_2)\mathbf{U}(a_2 - t_0)\wb(t_0) = \mathbf{U}(t - t_0)\wb(t_0)$ together with \eqref{evolution} to deduce that
	\begin{equation}\label{Idois}
		\begin{split}
			\Vert \mathbf{U}(t-a_2)\wb(a_2)\Vert_{\Sb([a_2,s])}&\lesssim 	\Vert\mathbf{U}(t-t_0)\wb(t_0)\Vert_{\Sb([a_2,s])} + \sumk\Vert \nabla [f_k(\vb)-f_k(\ub)]\Vert_{L_t^2L_x^{\frac{2d}{d+2}}([a_2,s]\times\R^d)}\\
			&\quad  +\Vert \nabla\mathbf{e}\Vert_{\Lb_t^2\Lb_x^{\frac{2d}{d+2}}}\\
		&\lesssim \sumk\Vert \nabla [f_k(\vb)-f_k(\ub)]\Vert_{L_t^2L_x^{\frac{2d}{d+2}}([a_2, s] \times\R^d)} + \epsilon.
	\end{split}
	\end{equation}
The quantity $\sumk\Vert \nabla [f_k(\vb)-f_k(\ub)]\Vert_{L_t^2L_x^{\frac{2d}{d+2}}([a_2, s] \times \mathbb{R}^d)}$ can therefore be controlled, as before, via a continuity argument. As a consequence, we obtain
$$
\sumk\Vert \nabla [f_k(\vb)-f_k(\ub)]\Vert_{L_t^2L_x^{\frac{2d}{d+2}}(I_2 \times \mathbb{R}^d)}\leq C_2\epsilon,
$$
which implies
$$
\Vert \mathbf{w} \Vert_{\mathbf{S}(I_2)}\leq C_2  \epsilon \quad \mbox{and} \quad\Vert \nabla \wb \Vert_{\Xb(I_2)}\leq  cA'+C_1 \epsilon+C_2 \epsilon. 
$$
By iterating this procedure, we get
$$
\sumk\Vert \nabla [f_k(\vb)-f_k(\ub)]\Vert_{L_t^2L_x^{\frac{2d}{d+2}}(I_n \times \mathbb{R}^d)}\leq C_n\epsilon,
$$
and
\begin{equation}\label{wk111}
\Vert \mathbf{w} \Vert_{\mathbf{S}(I_n)}\leq C_n  \epsilon \quad \mbox{and} \quad \Vert \nabla \wb \Vert_{\Xb(I_n)}\leq  cA'+\sum_{j=1}^{n}C_j \epsilon 
\end{equation}
provided that \eqref{v56} and \eqref{evolution} hold with $t_0$ replaced by $a_j$. Indeed, evaluating the integral equation \eqref{wint} at $t = a_{n+1}$ and proceeding as in \eqref{wk1}-\eqref{wk2}, we have
\begin{equation}\label{wk11}
			\begin{split}
			\Vert	U_k(t-a_{n+1}) w_k(a_{n+1})\Vert_{{S}(I)} &\lesssim \Vert	U_k(t-t_0) w_k(t_0)\Vert_{{S}(I)}+\Vert \nabla
				e_k \Vert_{L_{t}^2L_x^{\frac{2d}{d+2}}(I\times\mathbb{R}^d)} \\
& \quad +\Vert \nabla[f_k(\ub)-f_k(\vb)]\Vert_{L_t^{2}L_x^{\frac{2d}{d+2}}(I\times\R^d)} \\
&\leq  \varepsilon +\sum_{j=1}^{N_1} C_j \varepsilon
\end{split}
		\end{equation}
		and
\begin{equation}\label{wk21}
\begin{split}
			\sup_{t\in[t_0,a_{n+1}]}\left\|w_k(t)\right\|_{\dot{H}_x^1} &\lesssim \left\|w_k(t_0)\right\|_{\dot{H}_x^1}+\Vert \nabla
				e_k \Vert_{L_{t}^2L_x^{\frac{2d}{d+2}}(I\times\mathbb{R}^d)} \\
& \quad +\Vert \nabla[f_k(\ub)-f_k(\vb)]\Vert_{L_t^{2}L_x^{\frac{2d}{d+2}}([t_0,a_{n+1}]\times\R^d)} \\
&\leq  cA^{\prime}+\varepsilon +\sum_{j=1}^{n+1} C_j \varepsilon.
\end{split}
		\end{equation}
Therefore, by choosing $\epsilon_0$ sufficiently small, depending on $M$, $A$, $A'$, and $d$, we can proceed with the inductive argument. Finally, the estimates \eqref{v100} and \eqref{v101} follow directly from \eqref{v55}, \eqref{wk111}, and \eqref{wk21}.
	\end{proof}
	
The preceding result allows us to establish a useful continuity result.
	
	\begin{proposition}
		Let $\mathbf{v}_0\in\dot{\mathbf{H}}^1(\mathbb{R}^d)$ be such that $\Vert \mathbf{v}_0\Vert_{\dot{\mathbf{H}}^1}\leq A$ and let $\mathbf{v}$ be the solution of \eqref{SISTR} with maximal-lifespan time interval $\tilde{I}=(T^-,T^+)$. Let $\mathbf{u}_0^n\rightarrow\mathbf{v}_0$ in $\dot{\mathbf{H}}^1(\mathbb{R}^d)$, and let $\mathbf{u}_n$ be the corresponding solutions of \eqref{SISTR}, with maximal-lifespan time interval $I_n=(T_n^-,T_n^+)$. Then, $T^-\geq \limsup_{n\rightarrow\infty}T_n^-$, $T^+\leq \liminf_{n\rightarrow\infty}T_n^+$ and for each $t\in(T^-,T^+)$, $\mathbf{u}_n(t)\rightarrow\mathbf{v}(t)$ in $\dot{\mathbf{H}}^1(\mathbb{R}^d)$.
	\end{proposition}
	\begin{proof}
		Let $I\subset\subset \tilde{I}$, so that $\sup_{t\in I}\Vert\mathbf{u}\Vert_{\dot{\mathbf{H}}_x^1}=A<\infty$, $\Vert \mathbf{u}\Vert_{\mathbf{S}(I)}=M<\infty$. By setting $A' = 2A$, we can find $\epsilon_0 = \epsilon_0(M, A, 2A, d)$ such that the result of the previous lemma holds. Now, choosing $\epsilon < \epsilon_0$ and selecting $n$ sufficiently large so that $\Vert \mathbf{U}(t)[\mathbf{u}_0^n-\mathbf{v}]\Vert_{\mathbf{S}(I)}\leq\epsilon$, which is possible due to Strichartz estimates and the fact that $\mathbf{u}_0^n\rightarrow\mathbf{v}_0$ in $\dot{\mathbf{H}}^1(\mathbb{R}^d)$, we conclude that $\mathbf{u}_n$ exists on $I$. The conclusion $\mathbf{u}_n(t)\rightarrow\mathbf{v}(t)$ in $\dot{\mathbf{H}}^1(\mathbb{R}^d)$ for $t\in(T^-,T^+)$ follows immediately from the well-posedness theory.
	\end{proof}
	
	\subsection{Variational estimates}
Next we will establish some coercivity lemmas adapted to our setting.
	\begin{lemma}[Coercivity I]\label{coercivity1} Assume that $\mathbf{u}_0\in\dot{\mathbf{H}}^{1}(\mathbb{R}^d)$ and let $\mathbf{u}$ be a solution of \eqref{SISTR} with maximal existence interval $I$. Let $\boldsymbol{\psi}\in\mathcal{G}$ be a ground state. Suppose that 
		$$
		\E(\mathbf{u}_0)<(1-\tilde{\delta})\E(\boldsymbol{\psi}).
		$$
		\begin{itemize}
			\item[(i)] If 
			$$
			K(\mathbf{u}_0)<K(\boldsymbol{\psi}),
			$$
			then there exists $\tilde{\delta}_1=\tilde{\delta}_1(\tilde{\delta}) $ such that
			$$
			K(\mathbf{u}(t))<(1-\tilde{\delta}_1)K(\boldsymbol{\psi}),
			$$
			for all $t\in I$. 
			
			\item[(ii)]
			If 
			$$
			K(\mathbf{u}_0)>K(\boldsymbol{\psi}),
			$$
			then there exists $\tilde{\delta}_2=\tilde{\delta}_2(\tilde{\delta}) $ such that
			$$
			K(\mathbf{u}(t))>(1+\tilde{\delta}_2)K(\boldsymbol{\psi}),
			$$
			for all $t\in I$. 
		\end{itemize}
	\end{lemma}
	
	\begin{proof}
		From the energy conservation and the sharp inequality \eqref{360}, we deduce
		\begin{equation}\label{q28}
			K(\mathbf{u}(t))\leq \E(\mathbf{u}_0)+C_{\rm opt}K(\mathbf{u}(t))^{\frac{d}{d-2}},\quad \mbox{for all} \quad t\in I.
		\end{equation}
		Let $G(t)=K(\mathbf{u}(t))$, $a=\E(\mathbf{u}_0)$, $b=2C_{\rm opt}$ and $q={\frac{d}{d-2}}$. Using \eqref{optimal}, we get
		$$
		\gamma=(bq)^{-\frac{1}{q-1}}=\left(2C_{\rm opt}\frac{d}{d-2}\right)^{-\frac{d-2}{2}}=K(\psib).
		$$ 
Thus, Lemma \ref{lema comp} yields the desired result.
	\end{proof}
	
	\begin{lemma}[Coercivity II]\label{qCE} 
		Under hypothesis of Lemma \ref{coercivity1} we have
		\begin{itemize}
			\item[(i)] If 
			$$
			K(\mathbf{u}_0)<K(\boldsymbol{\psi}),
			$$
			then there exists $\delta'=\delta'(\tilde{\delta})>0$ such that
			\begin{equation}\label{coer2}
			K(\mathbf{u}(t))-\frac{2d}{d-2}P(\mathbf{u}(t))\geq \delta'K(\mathbf{u}(t)),
			\end{equation}
			for all $t\in I$. Moreover,  $\E(\ub)\geq 0$.
			\item[(ii)] If
			$$
			K(\mathbf{u}_0)>K(\boldsymbol{\psi}),
			$$
			then there exists $\delta''=\delta''(\tilde{\delta})>0$ such that
			$$
			K(\mathbf{u}(t))-\frac{2d}{d-2}P(\mathbf{u}(t))\leq -\delta''K(\mathbf{u}(t)),
			$$
			for all $t\in I$.
			
		\end{itemize}
	\end{lemma}
	
	\begin{proof}
		We will show item (i). The second one is proved in an analogous way. Using the sharp inequality \eqref{360} and Lemma \ref{coercivity1}, we deduce
		$$
		1-\frac{2d}{d-2}\frac{P(\mathbf{u}(t))}{K(\mathbf{u}(t))}\geq 1-\frac{2d}{d-2}C_{\rm opt}K(\mathbf{u}(t))^{\frac{2}{d-2}}=1-\left[\frac{K(\mathbf{u}(t))}{K(\boldsymbol{\psi})}\right]^{\frac{2}{d-2}}\geq 1-(1-\tilde{\delta}_1)^{\frac{2}{d-2}}=:\delta'.
		$$
		Multiplying both sides by $K(\mathbf{u}(t))$ we get the result. 
		
		Next we claim that the energy of the solution in non negative in case (i). Indeed, if $\E(\ub)\geq \E(\psib)>0$ (see Remark \ref{obs32}) the result is trivial. On the other hand, if $\E(\ub)<\E(\psib)$, since
$$
\E(\ub)= K(\mathbf{u}(t))-2P(\mathbf{u}(t))=\frac{2}{d}K(\mathbf{u}(t))+\frac{d-2}{d}\left[K(\mathbf{u}(t))-\frac{2d}{d-2}P(\mathbf{u}(t))\right],
$$
inequality \eqref{coer2} implies that $\E(\ub)\geq 0$.
  
	\end{proof}
	
	\begin{lemma}[Energy trapping]\label{qET} Let $\mathbf{u}_0\in \dot{\mathbf{H}}^1(\mathbb{R}^d)$ and $\mathbf{u}$ be the corresponding solution of \eqref{SISTR} with maximal existence interval $I$. If $\E(\mathbf{u}_0)\leq (1-\delta)\E(\boldsymbol{\psi})$ and $K(\mathbf{u}_0)\leq (1-\delta')K(\boldsymbol{\psi})$, then
		\begin{equation*}\label{g312}
			K(\mathbf{u}(t)) \sim \E(\mathbf{u}(t)),\quad \mbox{for all} \quad t\in I.
		\end{equation*}
	\end{lemma}
	\begin{proof}
		By \eqref{optimal} and $\E(\mathbf{u}_0)\leq (1-\delta)\E(\boldsymbol{\psi})$ we obtain
		\begin{equation*}
			\begin{split}
				\E(\mathbf{u}(t))&\leq K(\mathbf{u}(t))+2|P(\mathbf{u}(t))|\\
				&\leq K(\mathbf{u}(t))+ 2C_{\rm opt}|K(\mathbf{u}(t))|^{\frac{d}{d-2}}\\
				&\leq \left(1+2C_{\rm opt}[(1-\tilde{\delta}_1)K(\boldsymbol{\psi})^{\frac{2}{d-2}}]\right)K(\mathbf{u}(t)).
			\end{split}
		\end{equation*}
		On the other hand, 
		\begin{equation*}
			\begin{split}
				\E(\mathbf{u}(t))&\geq \frac{2}{d}K(\mathbf{u}(t))+\frac{d-2}{d}[K(\mathbf{u}(t))-\frac{2d}{d-2}P(\mathbf{u}(t))]\\
				&\geq \frac{2}{d}K(\mathbf{u}(t))+\frac{d-2}{d}\delta' K(\mathbf{u}(t))\\
				&=\frac{2}{d}\left(1+\frac{d-2}{d}\delta'\right)K(\mathbf{u}(t)).
			\end{split}
		\end{equation*}
		Combining both inequalities, we get the result.
	\end{proof}

	\subsection{Concentration-compactness-rigidity method}
	We start by setting $E_C$ to be the threshold
	\begin{equation*}
		E_C= \sup\left.\begin{cases}0<E<\E(\psib):\,\, \mbox{for every radial} \,\, \ub_0\in\Dz,\,\, \mbox{such that}\\
		K(\ub_0)<K(\psib)\,\,\mbox{and} \,\,\E(\ub_0)<E,\mbox{then the corresponding solution}\\
		\mbox{ to \eqref{SISTR} is global and satisfies}\,\,\sumk \Vert u_k\Vert_{S(\R)}<\infty.\end{cases}\right\},
	\end{equation*}
where $\psib$ is any ground state solution for system \eqref{SIST3}. By the small data theory and Lemma \ref{qET}, the number $E_C$ is well-defined and satisfies $0 < \tilde{\delta} \leq E_C \leq \E(\psib)$, where $\tilde{\delta} \sim \delta_0$ as in Corollary \ref{smalldata}. To prove Theorem \ref{MAIN}, we must establish that $E_C = \E(\psib)$. From now on, we assume $E_C < \E(\psib)$. The first step is to demonstrate the existence of a solution with the critical energy $E_C$. To this end, we recall the linear profile decomposition result from \cite{keraani}, which relies on an improved Sobolev inequality established in \cite{GMO}. Here, we present the version adapted to the Schrödinger propagator $U_k(t)$, $k=1,\ldots,l$, for bounded sequences in $\Dz$.
	
	\begin{lemma}[Linear profile decomposition]\label{LPD} 
		Let $(\vb_n)\subset\Dz$ be a radial sequence such that $\|\vb_n\|_{\dot{\mathbf{H}}^1}\leq A$. Assume that $\Vert U_k(t)v_{kn}\Vert_{S(\mathbb{R})}\geq\delta(A)>0$, for $k=1,\ldots,l$, where $\delta(A)$ is given in Proposition \ref{BCL}. Then, there exists a radial sequence $(\mathbf{V}_j)\subset \Dz$, a subsequence of $(\vb_n)$ and parameters $(\lambda_{j,n},\,t_{j,n})\in\R^+\times\R$ with 
		\begin{equation}\label{ortho}
			\frac{\lambda_{j,n}}{\lambda_{j',n}}+\frac{\lambda_{j',n}}{\lambda_{j,n}}+\frac{|t_{j,n}-t_{j',n}|}{\lambda_{j,n}^2}\rightarrow\infty,\quad \hbox{as }n\rightarrow\infty,\,j\neq j'.
		\end{equation}
		If $\mathbf{L}_j(x,t)(t,x)=\mathbf{U}(t)\mathbf{V}_j$, then given $\epsilon_0>0$, there exists $J=J(\epsilon_0)$ and a radial sequence $(\wb_n)\subset\Dz$ so that
		\begin{equation}\label{dec1}
			\vb_n=\sum_{j=1}^J\frac{1}{\lambda_{j,n}^{(d-2)/2}}\mathbf{L}_j\left(\frac{x}{\lambda_{j,n}},\frac{-t_{j_n}}{\lambda_{j,n}^2}\right)+\wb_n,
		\end{equation}
		with $\Vert U_k(t)w_{kn}\Vert_{S(\R)}\leq \epsilon_0$, for $k=1,\ldots,l$ and $n$ large. Moreover,
		\begin{equation}\label{dec2}
			\lim_{n\rightarrow\infty} K(\vb_n)=\sum_{j=1}^JK(\mathbf{V}_j)+\lim_{n\rightarrow\infty}K(\wb_n)\\
		\end{equation}
		and
		\begin{equation}\label{dec3}
			\lim_{n\rightarrow\infty} \E(\vb_n)=\lim_{n\rightarrow\infty}\sum_{j=1}^J \E(\mathbf{L}_j(-t_{j,n}/\lambda_{j,n}^2))+\lim_{n\rightarrow\infty}\E(\wb_n).
		\end{equation}
	\end{lemma}
	\begin{proof}
		The proof can be adapted from the one presented in \cite[Theorem 1.6]{keraani} and also \cite[Theorem 4.1]{visan}. We note that, in the radial case, it is not necessary to introduce a translation parameter.
	\end{proof}
	\begin{lemma}\label{lema49}
		Let $(\zb_0^n)\subset\Dz$ be a radial sequence, with $K(\zb_0^n)<K(\psib)$ and $\E(\zb_0^n)\rightarrow E_C$ with $\Vert \mathbf{U}(t)\zb_0^n\Vert_{\Sb(\R)} \geq \delta$ and $\delta$ as in Proposition \ref{BCL}. Let $(\mathbf{V}_j)$ be as in Lemma \ref{LPD}. Assume that one of the two hypothesis
		\begin{equation}\label{km410}
			\liminf_{n\rightarrow\infty}\E(\mathbf{L}_j(-t_{1,n}))/\lambda_{1,n}^2)<E_C
		\end{equation}
		or, after passing to a subsequence, we have that, $s_n=-t_1/\lambda_{1,n}^2$, $\E(\mathbf{L}_j(s_n))\rightarrow E_C$ and $s_n\rightarrow s_*\in[-\infty,\infty]$ and if $\mathbf{W}_1$ is the non-linear profile associated to $(\mathbf{V}_1, (s_n))$ we have that the maximal interval of existence of $\mathbf{W}_1$ is $I=\R$, $\Vert \mathbf{W}_1\Vert_{\Sb(\R)}<\infty$ and 
		\begin{equation}\label{km411}
			\liminf_{n\rightarrow\infty}\E(\mathbf{L}_j(-t_{1,n}/\lambda_{1,n}^2))=E_C.
		\end{equation}
		Then, up to a subsequence, for $n$ large, if $\zb_n$ is the solution associated to $\zb_0^n$, then $\zb_n$ is global and $\Vert \zb_n\Vert_{\Sb(\R)}<\infty$.
	\end{lemma}
	\begin{proof}
		The proof follows by adapting the one presented in \cite[Lemma 4.9]{KM}.
	\end{proof}
	
	\begin{proposition}\label{criticalsolution}[Existence of the critical solution]. Assume that $E_C<\mathcal{E}(\psib)$. Then, there exists a radially symmetric initial data $\ub_{0,C}\in \Dz$ satisfying
	\begin{equation}\label{o51}
			K(\ub_{0,C})<K(\psib)\quad \mbox{and}\quad \E(\ub_{0,C})=E_C
		\end{equation}
such that the corresponding solution $\ub_C$ of \eqref{SISTR} on a time interval $I_C\subset\R$ satisfies
		\begin{equation}\label{o512}
			 \sumk\Vert u_{kC}\Vert_{S(I_C)}=\infty.
		\end{equation}
	\end{proposition}
	\begin{proof}
	Suppose that $E_C<\E(\psib)$.	By definition of $E_C$, there exists a sequence of radial initial data $(\ub_0^n)\in\Dz$ such that
		$$
		\sup_{n\in\N}K(\ub_0^n)<K(\psib),\quad \E(\ub_0^n)\searrow E_C\quad\mbox{and}\quad \sup_{n\in\N}\E(\ub_0^n)<\E(\psib).
		$$
		Let $\ub^n$ be the corresponding sequence of solutions of \eqref{SISTR} with maximal-lifespan $I_n$ and $\ub^n(0)=\ub_0^n$, then $\Vert U_k(t) u_{k0}^n\Vert_{S(\R)}\geq \delta>0$ and  $\Vert u_k^n\Vert_{S(I_n)}=\infty$ for $\kl$. As $E_C < \E(\psib)$, Lemma \ref{qET} ensures the existence of a $\tilde{\delta}$ such that $K(\ub^n(t)) \leq (1 - \tilde{\delta})K(\psib)$ for large $n$ and $t \in I_n$. Next, we apply the Linear Profile Decomposition, Lemma \ref{LPD}, to write 
		\begin{equation}\label{ndec1}
			\ub_0^n=\sum_{j=1}^J\frac{1}{\lambda_{j,n}^{(d-2)/2}}\mathbf{L}_j\left(\frac{x}{\lambda_{j,n}},\frac{-t_{j_n}}{\lambda_{j,n}^2}\right)+\wb_n,
		\end{equation}
		\begin{equation}\label{ndec2}
			\lim_{n\rightarrow\infty} K(\ub_0^n)=\sum_{j=1}^JK(\mathbf{V}_j)+\lim_{n\rightarrow\infty}K(\wb_n),
		\end{equation}
		\begin{equation}\label{ndec3}
			\lim_{n\rightarrow\infty} \E(\ub_0^n)=\lim_{n\rightarrow\infty}\sum_{j=1}^J \E\left(\mathbf{L}_j\left(\frac{-t_{j,n}}{\lambda_{j,n}^2}\right)\right)+\lim_{n\rightarrow\infty}\E(\wb_n).
		\end{equation}
		Hence, \eqref{ndec2} implies that for $n$ large enough $K(\wb_n)\leq (1-\tilde{\delta}/2)K(\psib)$ and $K(\mathbf{V}_j)\leq (1-\tilde{\delta}/2)K(\psib)$. Also, \eqref{ndec3} together with Lemma \ref{qCE} leads us to 
		\begin{equation}
			\lim_{n\rightarrow\infty}\E\left(\mathbf{L}_1\left(\frac{-t_{1,n}}{\lambda_{1,n}^2}\right)\right)\leq E_C
		\end{equation}
and, by applying Lemma \ref{lema49}, we conclude the equality.
		
		Let $\Wb_1$ be the nonlinear profile associated with $(\mathbf{L}1, (s_n))$, where $s_n = -t{1,n}/\lambda_{1,n}^2$. Given \eqref{ndec3} and the fact that $\E(\mathbf{L}_1(-s_n)) \to 0$, we obtain $J = 1$. Combined with Lemma \ref{qCE}, this implies that $\mathbf{V}_j = 0$ for $j = 2, \dots, J$ and $K(\wb_n) \to 0$. Consequently, \eqref{ndec1} becomes
		$$
		\ub_0^n=\frac{1}{\lambda_{1,n}^{\frac{d-2}{2}}}\mathbf{L}_1\left(s_n,\frac{x}{\lambda_{1,n}}\right)+\wb_n.
		$$
		
		Define $\vb_0^n = \lambda_{1,n}^{\frac{d-2}{2}} \ub_0^n(\lambda_{1,n}x)$ and $\tilde{\wb}_n=\lambda_{1,n}^{\frac{d-2}{2}}\wb_n(\lambda_{1,n}x)$. Thus,
		$$\vb_0^n=\mathbf{L}_1(s_n)+\tilde{\wb}_n,\quad K(\tilde{\wb}_n)\rightarrow 0.$$
Now, let $I_1$ be the maximal-lifespan of $\Wb_1$. Note that in this case, by Definition \ref{non-linear profile} and Lemma \ref{qET}  we have $\lim_{n\rightarrow\infty}K(\Wb_1(s_n))=\lim_{n\rightarrow\infty}K(\ub_0^n)<K(\psib)$ and $\lim_{n\rightarrow\infty}\E(\Wb_1(s_n))=E_C$. If we fix $\tilde{s}\in I_1$, by conservation of the energy $\E(\Wb_1(\tilde{s}))=E_C$. In addition, using Lemma \ref{qET} once again, we have $K(\Wb_1(\tilde{s}))<K(\psib)$. If $\Vert \mathbf{W}_1\Vert_{S(I_1)}<\infty$, then Lemma \ref{standblow} gives us $I_1=\R$, which is a contradiction with Lemma \ref{lema49}. Therefore, we must have $\Vert \Wb_1\Vert_{S(I_1)}=\infty$. So after a translation in time, to make $\tilde{s}=0$, we set $\ub_C=\Wb_1$.  
	\end{proof}
	Next, we prove the precompactness in $\Dz$ of the flow of the critical solution $\ub_C$.
	\begin{proposition}\label{precompactness}
		Assume that $\ub_C$ is as is Proposition \ref{criticalsolution} and $\Vert \ub_C\Vert_{S(I_+)}=\infty$, where $I_+=(0,\infty)\cap I_C$. Then there exist $x(t)\in\R^d$ and $\lambda(t)\in\R^+$ such that
		\begin{equation}\label{K}
			\mathcal{K}=\left\{\vb(t,x)=\frac{1}{\lambda(t)^{\frac{d-2}{2}}}\ub_C\left(t, \frac{x-x(t)}{\lambda(t)}\right); \,\, t\in I_+ \right\}
		\end{equation}
	 is pre-compact in $\Dz$. A similar result holds if $\Vert \ub_C\Vert_{S(I_-)}$, where $I_-=(-\infty,0)\cap I_C$. 
	\end{proposition}
	\begin{proof} As the proof is independent of the nonlinearities, we offer only a brief outline. For details, the reader may refer to \cite[Proposition 4.2]{KM}. Assume that $\mathcal{K}$ is not is pre-compact in $\Dz$. Then, there exists $\eta_0>0$ and a sequence of time $(t_n)\subset(0,\infty)$ such that, for all $\lambda_0\in\R^+$ and $x_0\in\R^d$, we have
		\begin{equation}\label{km415}
			\left\Vert \frac{1}{\lambda_0^{\frac{d-2}{2}}}\ub_C\left(t_n, \frac{x}{\lambda_0}\right)-\ub_C(t_m,x)\right\Vert_{\dot{\mathbf{H}}_x^1}\geq \eta_0,\quad n\neq m.
		\end{equation}
By applying Lemma \ref{LPD} to $\vb_0 = \ub_C(t_n)$ with $\epsilon_0 > 0$, we again find that $J = 1$, following the same argument as in the previous proof. Thus,
		\begin{equation}\label{km416}
			\ub_C(t_n)=\frac{1}{\lambda_{1,n}^{\frac{d-2}{2}}}\mathbf{L}_1\left(\frac{-t_{1,n}}{\lambda_{1,n}},\frac{x}{\lambda_{1,n}^2}\right)+\wb_n,\quad K(\wb_n)\rightarrow 0.		
		\end{equation}
		Now, since $s_n=\frac{-t_{1,n}}{\lambda_{1,n}^2}$ is bounded, up to a subsequence
		\begin{equation}\label{sn}
		s_n\rightarrow t_0\in\R.
		\end{equation}
		Then, by \eqref{km415} and \eqref{km416} and uniqueness of the limit $\eqref{sn}$, we reach a contradiction by choosing suitable $\lambda_0\in\R^+$.
	\end{proof}
		\begin{remark}
		We can assume that the frequency scale function $\lambda(t)$ is bounded from below, i.e., $\lambda(t) \geq A_0 > 0$, as detailed in \cite[Section 5]{KM}. Furthermore, $\lambda(t)$ can be adjusted so that $\inf_{t \in I_+} \lambda(t) \geq 1$ for at least half of the maximal lifespan, following the arguments in \cite[Section 4]{KV}.
	\end{remark}
	
	Next we will prove that the critical solution must be global.
	\begin{proposition}\label{kg53} Suppose $E_C < \E(\psib)$ and let $\ub_C$ be the solution established in Theorem \ref{criticalsolution}. Then, $I_C = \mathbb{R}$.
	\end{proposition}
	\begin{proof}
		Let $I_C = (-T_C^-, T_C^+)$. We will prove the result for $T_C^+$, with the case for $T_C^-$ following similarly. Assuming $T_C^+ < \infty$, we claim that
		\begin{equation}\label{lambda}
			\lambda(t)\rightarrow\infty,\quad \hbox{as} \quad t\rightarrow T_C^+.
		\end{equation}
		Otherwise, there is a sequence $t_i\uparrow T_C^+$ such that $\lambda(t_i)\rightarrow\lambda_0<\infty$ such that, by the compactness of $\overline{\mathcal{K}}$, the sequence   $\vb_i:=\lambda(t_i)^{-\frac{d-2}{2}}\ub_C\left(t_i,\lambda(t_i)^{-1}x\right)\rightarrow \vb$ in $\Dz$. Now, since $\inf_{t\in [0,T_C^+]}\lambda(t)\geq 1$, and $\lambda_0>0$, we conclude that
		 $\ub_C\left(t_i, x\right)=\lambda(t_i)^{\frac{d-2}{2}}\vb_i(t_i,\lambda(t_i)x)\rightarrow\lambda_0^{\frac{d-2}{2}}\vb(T_C^+,\lambda_0x)$. Let $h(t,x)$ be the solution of \eqref{SISTR} with initial data $\lambda_0^{\frac{d-2}{2}}\vb(T_C^+,\lambda_0x)$, in a interval $(T_C^+-\delta, T_C^++\delta)$, with $\Vert h\Vert_{S((T_C^+-\delta,T_C^++\delta))}<\infty$  and a sequence of solutions $\mathbf{h}_i(t,x)$ with initial data at $T_C^+$ equal to $\ub_C(t_i)$ with
		$$
		\sup_i \Vert \mathbf{h}_i\Vert_{S(T_C^+-\delta/,T_C^++\delta/)}<\infty.
		$$
However, $\ub_C(t+t_i-T_C^+)=\mathbf{h}_i(t)$, which contradicts Lemma \ref{standblow}, since $T_C^+<\infty$. 
		
		Now, fix $\varphi\in C_0^{\infty}$ to be radial, equal $1$ for $|x|\leq 1$ and equal $0$ for $|x|\geq 2$ and set $\varphi_R(x)=\varphi(x/R)$. Define
		\begin{equation*}
			X_R(t)=\sumk \int \varphi_R(x)|\uck|^2dx,\quad t\in[0,T_C^+).
		\end{equation*}
		Thus, applying Hölder's inequality, Hardy's inequality, and \eqref{o51}, we get
		\begin{equation}\label{km58}
			|X'_R|\leq \frac{2}{R}\left| \IM \sumk\int \nabla \varphi_R \nabla\uck\cdot \bar{u}_{ck}dx.\right|\lesssim K(\ub_C)^{\frac{1}{2}}\left(\sumk \int \frac{|\uck|^2}{|x|^2}\right)^{\frac{1}{2}}\lesssim K(\psib).
		\end{equation}
Let $\vb$ be defined as in \eqref{K}. For any $R>0$, using a change of variables yields, for each $k=1,\ldots,l$,
		\begin{equation*}
			\begin{split}
				\int_{|x|<R}|\uck(t,x)|^2dx&=\lambda(t)^{-2}\int_{B(0,\epsilon R\lambda(t)}|v_k(t,z)|^2dz\\
				&\quad+\lambda(t)^{-2}\int_{B(0,R\lambda(t))\backslash B(0,\epsilon R\lambda(t))}|v_k(t,z)|^2dz\\
				&:=A+B	
			\end{split}
		\end{equation*}
Moreover, by Hölder's inequality and Sobolev embedding 
		$$
		A\leq \lambda(t)^{-2}(\epsilon R\lambda(t))^2 \Vert v_k\Vert_{L_x^\frac{2d}{d-2}}^2\lesssim \epsilon^2R^2K(\psib),
		$$
which can be controlled by choosing $\epsilon>0$ small enough. Also, by \eqref{lambda} and the pre-compactness of $\Km$, we have
		\begin{equation*}
			B\leq R^2\Vert v_k\Vert_{L^{\frac{2d}{d-2}}(|x|\geq \epsilon R\lambda(t))}^2\rightarrow 0,\quad \hbox{as }t\rightarrow T_C^+.
		\end{equation*}
		Therefore, for all $R>0$,
		\begin{equation}\label{km59}
			\sumk\int_{|x|<R}|\uck|^2dx\rightarrow 0\quad\hbox{as }t\rightarrow T_C^+,
		\end{equation}	
which implies that $\lim_{t\to T_C^+}X_R(t)=0$.
Therefore, in view of the \eqref{km58}, the Fundamental Theorem of Calculus implies
		$$
		X_R(0)\lesssim  T_C^+K(\psib)\quad \mbox{and}\quad X_R(t)\lesssim (T_C^+-t)K(\psib).
		$$
Therefore, making $R\rightarrow\infty$ we obtain that $\ub_C(t)\in\Lb^2(\R^d)$, for all $t\in[0,T_C^+)$, and $\Vert\ub_C(t)\Vert_{\Lb_2}^2\lesssim (T_C^+-t)K(\psib)$. Finally, by mass conservation, we must have $\ub_C\equiv 0$ which contradicts \eqref{o51}.	
	\end{proof}
	We now finalize the contradiction argument by showing that the solution $\ub_C$ cannot exist. To this end, we note that precompactness of the flow, modulo the invariances of the equation, occurs only for the trivial solution.
	\begin{proposition}\label{inexistence}
		Suppose $E_C < \E(\psib)$ and let $\ub_C$ be the solution established in Theorem \ref{criticalsolution}. Then $\ub_C\equiv 0$. 
	\end{proposition}
	\begin{proof}
We begin by noting that, using Sobolev embedding and the precompactness of $\mathcal{K}$, for each $\epsilon > 0$ there exists $R(\epsilon) > 0$ such that for all $t \geq 0$
		\begin{equation}\label{km510}
			\sumk \int_{|x|>R(\epsilon)}|\nabla \uck|^2+|\uck|^{\frac{2d}{d-2}}\leq\epsilon.
		\end{equation}
		Let $\phi \in C_0^\infty(\mathbb{R}^d)$ be a radial function such that $\phi(x) = |x|^2$ for $|x| \leq 1$ and $\phi(x) = 0$ for $|x| \geq 2$. Define
		$$
		V_R(t)=\sumk\int \frac{\alpha_k^2}{\gamma_k}|\uck|^2R^2\phi\left(\frac{x}{R}\right)dx.
		$$
Then, using \eqref{o51}, Hardy’s inequality, and Hölder’s inequality, we obtain
		\begin{equation}\label{V'2}
			|V'_R(t)|\lesssim 2R  \sumk \alpha_k \IM\int\left|\phi'\left(\frac{x}{R}\right)\bar{u}_{ck}\nabla\uck dx\right|\lesssim R^2K(\ub_C)\lesssim R^2,
		\end{equation}
for all $t \geq 0$ and $R > 0$. Moreover, applying Lemma \ref{lemma22} and noting that for $|x| \leq R$ we have $\partial_i \partial_j \phi(x) = 2 \delta_{ij}$, $\Delta \phi(x) = 2d$, and $\Delta^2 \phi(x) = 0$, we deduce
		\begin{equation}
			\begin{split}
				V_R''(t)&=4\sum_{1\leq j,m\leq d}\RE \int \frac{\partial^2\phi}{\partial x_m\partial x_j}\left[\sumk \gamma_k\partial_{x_j}\bar{u}_{ck}\partial_{x_m}\uck\right]dx\\
				&\quad -\int\Delta^2\phi\left(\sumk \gamma_k |\uck|^2\right)dx
				- \frac{8}{d-2}\int \Delta\phi F(\ub_C)dx+8\sumk \int_{|x|>R}\gamma_k|\nabla \uck|^2dx\\	
				&\quad -8\sumk \int_{|x|>R}\gamma_k|\nabla \uck|^2dx
				+\frac{16d}{d-2}\int_{|x|>R}F(\ub_C)dx-\frac{16d}{d-2}\int_{|x|>R}F(\ub_C)dx\\
				&= 8\left[K(\ub_C)-\frac{2d}{d-2}P(\ub_C)\right]+O\left(\sumk \int_{R\leq|x|\leq 2R(\epsilon)}|\nabla \uck|^2+|\uck|^{\frac{2d}{d-2}} \right).\\
						\end{split}
		\end{equation}
		Observe that by selecting $R = R(\epsilon)$, we can control the last term using \eqref{km510}. Hence, by choosing $\epsilon$ sufficiently small and applying conservation of energy, Lemma \ref{qCE}, and Lemma \ref{qET}, we obtain
		\begin{equation}\label{h41333}
			V_R''(t)\gtrsim K(\ub_C)\gtrsim \E(\ub_{C0}).
		\end{equation}
Integrating in $t$ and using \eqref{V'2} along with the Fundamental Theorem of Calculus, we obtain
		$$  
		t\E(\ub_{C0})\lesssim V_R'(t)-V_R'(0) \lesssim |V_R'(t)-V_R'(0)|\lesssim R^2(\epsilon).
		$$ 
Dividing both sides by $t > 0$ and letting $t \to \infty$, it follows that $\E(\ub_{C0}) = 0$. By conservation of energy, this implies $\ub_C \equiv 0$.
	\end{proof}
	
We are now ready to prove our global well-posedness and scattering result.
	
	\begin{proof}[Proof of Theorem \ref{MAIN}] By Theorem \ref{inexistence}, it follows that $E_C = \mathcal{E}(\psib)$. Let $\ub_0 \in \Dz$ be radial and satisfy \eqref{mainE} and \eqref{main K}, thus the corresponding solution is global with $\sum_{k=1}^l \|u_k\|_{S(\mathbb{R})} < \infty$. For the scattering part, observe that, by applying Lemma \ref{strichartz} and Lemma \ref{consequences}, we obtain that
		$$\left\Vert \int_{t}^{\infty}U_k(t-s)f_k(\ub)ds\right\Vert_{\Dz}\rightarrow 0\quad \hbox{as }t\rightarrow\infty.$$ 
Then, for each $k = 1, \ldots, l$, if we define 
$$
u_k(t)=U_k(t-a)u_k(0)+\int_a^tU_k(t-s)f_k(\ub)ds
$$ 
and 
$$
u_k^+=U_k(-a)u_k(0)+\int_a^\infty U_k(-s)f_k(\ub)ds,
$$
we obtain the desired result.
	\end{proof}

	\section{Blow-up in finite time}\label{secbl}
To begin, we note that the well-posedness theory in $\mathbf{H}^1(\mathbb{R}^d)$ stated in Theorem \ref{LWP2} can be established using an argument similar to that of Theorem \ref{LWP}, with the addition of the norm $\mathbf{X}(I)$ in the definition of the ball for the contraction mapping argument (see \cite[Theorem 2.5]{KM}). Moreover, if the initial data $\ub_0$ is radially symmetric, then the corresponding solution $\ub$, defined on its maximal existence interval $I$, is also radially symmetric.
		
As mentioned in the Introduction, we follow the strategy from \cite{inui} and work with the function $\mR$ defined in \eqref{R}. Furthermore, if we assume $\phi$ is radially symmetric, a simple calculation (see, e.g., \cite[Lemma 2.9]{kavian}) allows us to express $\mR'(t)$ as
	\begin{equation}\label{R'}
		\mR'(t)=4\int\phi''\left(\sumk \gamma_k|\nabla u_k|^2\right)dx-\int\Delta^2\phi\left(\sumk \gamma_k|u_k|^2\right)dx-\frac{8}{d-2}\RE\int\Delta\phi F(\ub)dx.
	\end{equation}	
	Now, we introduce the functional
	\begin{equation*}
		\tau(\ub(t))=K(\ub(t))-\frac{2d}{d-2}P(\ub(t)).
	\end{equation*}
This functional is often referred to as the "Pohozaev" functional due to its close connection with the Pohozaev identities (see \cite[Lemma 4.3]{noguerap}). By employing the energy definitions in \eqref{EQ}, we can rewrite
	\begin{equation}\label{411}
		\tau(\ub(t))=\frac{d}{d-2}E(\ub(t))-\frac{2}{d-2}K(\ub(t))-\frac{d}{d-2}L(\ub(t))
	\end{equation}	
	where $L(\ub)$ is defined in \eqref{KLP}. The following lemma establishes a key property of the Pohozaev functional, namely that it is strictly negative for solutions of \eqref{SISTQ}.
	\begin{lemma}\label{lema delta}
		Under the assumptions of Theorem \ref{blow4}, there exists $\delta>0$ such that
		$$
		\tau(\ub(t))\leq -\delta<0,\quad t\in I.
		$$
	\end{lemma}
	\begin{proof}
		We adopt the approach from \cite[Theorem 1.3]{DUWU}. We begin by observing that \eqref{E} and \eqref{36} imply
		\begin{equation}\label{412}
			K(\psib)=\frac{d}{2}\mathcal{E}(\psib).
		\end{equation}
Moreover, from the sharp inequality \eqref{360} and conservation of the energy, we have
		\begin{equation}\label{413}
			K(\ub)=E(\ub_0)-L(\ub)+2P(\ub)\leq E(\ub_0)+2|P(\ub)|\leq E(\ub_0)+2C_{\rm opt}K(\ub)^{\frac{d}{d-2}}.\\
		\end{equation}
Define $f(r) = E(\ub_0) - r + 2C_{\rm opt} r^{\frac{d}{d-2}}$ for $r > 0$, and let $G(t) = K(\ub(t))$. From \eqref{413}, it follows that
		$$
		f\circ G(t)=E(\ub_0)-K(\ub(t))+2C_{\rm opt}K(\ub(t))^{\frac{d}{d-2}}\geq 0.
		$$  
Applying Lemma \ref{lema comp} with $\delta=0$, $a=E(\ub_0)$, $b=2C_{\rm opt}$, $q=\frac{d}{d-2}$ and $\gamma=\left(\frac{2C_{\rm opt}d}{d-2}\right)^{-\frac{d-2}{2}}$, from \eqref{optimal}, we obtain
		$$
		a<\left(1-\frac{1}{q}\right)\gamma \Leftrightarrow E(\ub_0)<\mathcal{E}(\psib),
		$$
and, from relation \eqref{412}, it follows that
		$$
		G(0)>\gamma \Leftrightarrow K(\ub_0)> \frac{d}{2}\mathcal{E}(\psib)=K(\psib).
		$$
		Hence, considering \eqref{41} and \eqref{42}, we can apply Lemma \ref{lema comp} to obtain
		\begin{equation}\label{414}
			K(\ub(t))>K(\psib),\quad  \quad t\in I.
		\end{equation}
		The assumption \eqref{41}, combined with the conservation of energy and \eqref{412}, yield
		$$
		\frac{d}{2}E(\ub(t))=\frac{d}{2}E(\ub_0)<\frac{d}{2}\mathcal{E}(\psib)=K(\psib)<K(\ub(t)), \quad t\in I
		$$
		and consequently, from \eqref{411},
		\begin{equation*}
			\tau(\ub(t))<0,\quad t\in I.
		\end{equation*}
		
		Now, let us show that there is $\theta>0$, such that
		\begin{equation}\label{416}
			\tau(\ub(t))<-\theta K(\ub(t)),\quad t\in I.
		\end{equation} Indeed, if $E(\ub_0)\leq 0$ the result follows immediately. On the other hand, suppose $E(\ub_0)>0$ and that \eqref{416} does not hold. Then, there exist sequences $(t_m)\subset I$ and $\theta_m\rightarrow 0$ such that
		$$
		\frac{-\theta_m }{d-2} K(\ub(t_m))\leq \tau(\ub(t_m))<0,
		$$
		which implies
		\begin{equation*}
			E(\ub(t_m))=\frac{d-2}{d}\tau(\ub(t_m))+\frac{2}{d}K(\ub(t_m))+L(\ub(t_m))\geq (1-\theta_m)\frac{2}{d}K(\ub(t_m)).\\
		\end{equation*}
		Using the conservation of energy again, \eqref{412} and \eqref{414} imply that
		\begin{equation*}
			E(\ub_0)=E(\ub(t_m))\geq(1-\theta_m)\frac{2}{d}K(\ub(t_m))
			>(1-\theta_m)\frac{2}{d}K(\psib)
			\geq (1-\theta_m)\mathcal{E}(\psib).
		\end{equation*}
Taking the limit as $m\rightarrow\infty$ leads to a contradiction with \eqref{41}. Therefore, the result follows from \eqref{414} and \eqref{416} by setting $\delta=\theta K(\psib)$.
\end{proof}
Another tool to prove Theorem \ref{blow4} is the following straightforward result. Since its proof involves a direct computation, we omit the details. For a similar statement in dimension 4, the reader may refer to \cite[Lemma 4.2]{HP}.
		\begin{lemma}\label{lema chi}
For $x\in\mathbb{R}^d$ and some constant $c>0$, define $\chi(x)\in C^{\infty}(\mathbb{R}^d)$ a radial function satisfying
		\begin{equation*}
			\chi(x)=\left\{\begin{array}{ll}
				|x|^2, & 0\leq |x|\leq 1,\\
				c, & |x|\geq 3.
			\end{array}\right.
		\end{equation*}
		Assume also that $\chi''(x)\leq 2$ and $0\leq \chi'(x)\leq 2r$, for all $r\geq 0$. Let $\chi_R(x)=R^2\chi(x/R)$. Then
		\begin{itemize}
			\item[(i)] If $|x|\leq R$,
			\begin{equation}\label{48}
				\Delta\chi_R(x)=2d\quad\mbox{and}\quad\Delta^2\chi_R(x)=0.
			\end{equation}

			\item[(i)] If $|x|\geq R$,
			\begin{equation}\label{49}
				\Delta\chi_R(x)\leq C\quad\mbox{and}\quad|\Delta^2\chi_R(x)|\leq\frac{C}{R^2},
			\end{equation}
			where $C$ is a constant independent of $R$.
		\end{itemize}
	\end{lemma}
	
Finally, we prove our blow-up result.
	\begin{proof}[Proof of Theorem \ref{blow4}]
		Let $I = (T^-, T^+)$ be the maximal lifespan interval of $\ub$. We will prove only that $T^+ < \infty$, since the case $T^- > -\infty$ can be handled in a similar manner. By choosing $\phi(x) = \chi_R(x)$ as defined in Lemma \ref{lema chi} and using \eqref{R} and \eqref{R'}, we obtain
		$$
		\mathcal{R}(t)=2\sumk \IM\int\alpha_k \nabla\chi_R\cdot\nabla u_k \bar{u}_kdx
		$$
		and
		\begin{equation*}
			\begin{split}
				\mathcal{R}'(t)&=8\tau(\ub)+4\int(\chi_R''-2)\left(\sumk \gamma_k |\nabla u_k|^2\right)dx-\int\Delta^2\chi_R\left(\sumk \gamma_k|u_k|^2\right)dx\\
				&\quad-\frac{8}{d-2}\hbox{Re}\int(\Delta\chi_R-2d)F(\ub)dx\\
				&=:8\tau(\ub)+\mathcal{R}_1(t)+\mathcal{R}_2(t)+\mathcal{R}_3(t).
			\end{split}
		\end{equation*}
By construction, since $\chi_R'' \leq 2$ for all $r \geq 0$, it follows that $\mathcal{R}_1 \leq 0$. Then, applying mass conservation together with \eqref{48}–\eqref{49}, we obtain
		$$
		\mathcal{R}_2(t)\leq\int|\Delta^2\chi_R|\left(\sumk \gamma_k|u_k|^2\right)dx\leq CR^{-2}\int_{\{|x|\geq R\}}\left(\sumk \gamma_k|u_k|^2\right)dx\leq CR^{-2}Q(\ub_0).
		$$
		Moreover, from  \eqref{48} and Lemma \ref{lemma22}, we have
		\begin{equation*}
			\begin{split}
				\mathcal{R}_3&=-\frac{8}{d-2}\hbox{Re}\int_{\{|x|\geq R\}}(\Delta\chi_R-2d)F(\ub)dx\\
				&\leq C\left(\frac{1}{d-2}\right)\int_{\{|x|\geq R\}}|\mathrm{Re} F(\ub)|dx\\
				&\leq C\left(\frac{1}{d-2}\right)\int_{\{|x|\geq R\}}\sumk |u_k|^{\frac{2d}{d-2}}  dx\\
				&=C\left(\frac{1}{d-2}\right)\sumk \Vert u_k\Vert_{L^{\frac{2d}{d-2}}(|x|\geq R)}^{\frac{2d}{d-2}}.
			\end{split}
		\end{equation*}
Next we recall the radial the radial Gagliardo-Nirenberg inequality (see, for instance, \cite[page 323]{ogawa})
		$$
		\int_{\{|x|\geq R\}}|f|^{\frac{2d}{d-2}}\leq CR^{-\frac{2(d-1)}{d-2}}\Vert f\Vert^{\frac{2(d-1)}{d-2}}_{L^2(|x|\geq R)}\Vert\nabla f\Vert_{L^2(|x|\geq R)}^{\frac{2}{d-2}}.
		$$
		Then, for $d>3$ and $\epsilon>0$, by Young's inequality, we obtain
		\begin{equation*}
			\begin{split}
				\mathcal{R}_3&\leq C\left(\frac{1}{d-2}\right)\sumk R^{-\frac{2(d-1)}{d-2}}\Vert u_k\Vert^{\frac{2(d-1)}{d-2}}_{L^2(|x|\geq R)}\Vert\nabla u_k\Vert_{L^2(|x|\geq R)}^{\frac{2}{d-2}}\\
				&\leq C_\epsilon R^{-\frac{2(d-1)}{d-3}}\left(\frac{1}{d-2}\right)M(\ub_0)^{\frac{2(d-1)}{d-3}}+\left(\frac{2\epsilon}{d-2}\right) K(\ub),\\
			\end{split}
		\end{equation*}
		where $C_\epsilon$ is a constant depending on $\epsilon$, $\alpha_k$, $\gamma_k$ and $\sigma_k$. Also note, from \eqref{411}, that
		\begin{equation*}
			\left(\frac{2\epsilon}{d-2}\right) K(\ub)\leq \frac{d}{d-2}\epsilon E(\ub_0)-{\epsilon}\tau(\ub).
		\end{equation*}
		Putting together the above estimates, we obtain
		\begin{equation}\label{417}
			\mathcal{R}'(t)\leq \left(8-\epsilon\right)\tau(\ub)+CR^{-2}M(\ub_0)+C_\epsilon R^{-\frac{2(d-1)}{d-3}}\frac{1}{d-2}M(\ub_0)^{\frac{2(d-1)}{d-3}}+\frac{d}{d-2}\epsilon |E(\ub_0)|.
		\end{equation}
		Therefore, for $\epsilon\in(0,1)$, Lemma \ref{lema delta} yields
		\begin{equation*}
			\mathcal{R}'(t)\leq -\left(8-\epsilon\right)\delta+CR^{-2}M(\ub_0)+C_\epsilon R^{-\frac{2(d-1)}{d-3}}M(\ub_0)^{\frac{2(d-1)}{d-3}}+\frac{d}{d-2}\epsilon |E(\ub_0)|.
		\end{equation*}
		Therefore, by fixing $R$ large enough and then $\epsilon$ small enough, we obtain $\mathcal{R}'(t)\leq -2\delta$. Integrating this inequality over $[0,t)$ leads to
		\begin{equation}\label{418}
			\mathcal{R}(t)\leq -2\delta t+\mathcal{R}(0).
		\end{equation}
For $d=3$, we have the estimate 
$$
\mathcal{R}_3\leq C R^{-4}M(\ub_0)^2K(\ub),
$$ 
which, by following the same argument, yields
	$$
	\mathcal{R}'(t)\leq -\left(8-{CR^{-4}M(\ub_0)^2}\right)\delta+CR^{-2}M(\ub_0)+CR^{-4}M(\ub_0)|E(\ub_0)|,
	$$
and this also implies \eqref{418} provided that $R$ is chosen sufficiently large.
	
			On the other hand, by Hölder's inequality,
		\begin{equation}\label{419}
			\begin{split}
				|\mathcal{R}(t)| &\leq 2R\sumk \alpha_k\int |\chi'(|x|/R)||\nabla u_k||u_k| dx\\
				&\leq CR\sumk \alpha_k(\Vert u_k\Vert_{L^2}\Vert \nabla u_k\Vert_{L^2})\\
				&\leq CR M(\ub_0)^{1/2}K(\ub(t))^{1/2}.\\
			\end{split}
		\end{equation}
		By combining \eqref{418} with the fact that $\mathcal{R}(0)/\delta < T_0$ can be ensured by taking $T_0$ sufficiently large, we conclude that
		\begin{equation}\label{420}
			\mathcal{R}(t)\leq-\delta t<0,\quad t\geq T_0.
		\end{equation}
		Thus, by applying \eqref{419} and \eqref{420}, we deduce that 
$$
\delta t\leq-\mathcal{R}(t)=|\mathcal{R}(t)|\leq CRM(\ub_0)^{1/2}K(\ub)^{1/2},
$$ 
which implies that, for some positive constant $C_0$,
		\begin{equation}\label{421}
			K(\ub(t))\geq C_0t^2,\quad t\geq T_0.
		\end{equation}
		Now, since $\epsilon$ can be chosen arbitrarily small, it follows from \eqref{417} and \eqref{411} that
		\begin{equation}\label{422}
			\mathcal{R}'(t)\leq\left(\frac{8d}{d-2}\right)E(\ub_0)- \left(\frac{15}{d-2}\right)K(\ub)+ CR^{-2}M(\ub_0)+ C_\epsilon R^{-\frac{2(d-1)}{d-3}}M(\ub_0)^{\frac{2(d-1)}{d-3}},
		\end{equation}
where we have once again used the conservation of energy and the fact that $L(\ub) \geq 0$. Now, we can choose $T_1 > T_0$ such that
		$$
		C_0\left(\frac{15}{2(d-2)}\right)T_1^2\geq\left(\frac{8d}{d-2}\right)E(\ub_0)+CR^{-2}M(\ub_0)+C_\epsilon R^{-\frac{2(d-1)}{d-3}}M(\ub_0)^{\frac{2(d-1)}{d-3}}.
		$$
		From \eqref{421} and \eqref{422}, we obtain
		$$
		\mathcal{R}'(t)\leq -\left(\frac{15}{2(d-2)}\right)K(\ub(t)), \quad t>T_1.
		$$
		Integrating over $[T_1, t)$, this yields
		$$
		\mathcal{R}(t)\leq -\left(\frac{15}{2(d-2)}\right)\int_{T_1}^tK(\ub(s))ds.
		$$
		Combining with \eqref{419}, we derive
		\begin{equation}\label{423}
			\left(\frac{15}{2(d-2)}\right)\int_{T_1}^tK(\ub(s))ds\leq -\mathcal{R}(t)\leq |\mathcal{R}(t)|\leq CRM(\ub_0)^{1/2}K(\ub(t))^{1/2}.
		\end{equation}
Now set $\eta(t):=\displaystyle\int_{T_1}^tK(\ub(s))ds$ and $A:=\frac{15^2}{4(d-2)^2C^2R^2M(\ub_0)} $. The previous inequality implies
		$$
		A\leq\frac{\eta'(t)}{\eta(t)^2},
		$$
Choosing $T' > T_1$ and integrating over $[T', t)$, we obtain
		$$
		A(t-T')\leq \int_{T'}^t\frac{\eta'(s)}{\eta(s)^2}ds=\frac{1}{\eta(T')}-\frac{1}{\eta(t)}\leq \frac{1}{\eta(T')},
		$$
which implies
		$$
		0<\eta(T')\leq \frac{1}{A(t-T')}.
		$$
If $T^+ = \infty$, letting $t \to \infty$ leads to a contradiction.
	\end{proof}
	
	\subsection*{Acknowledgment}
L.G.F. was partially supported by Coordenação de Aperfeiçoamento de Pessoal de
Nível Superior - CAPES, Conselho Nacional de Desenvolvimento Científico e Tecnológico - CNPq
and Fundação de Amparo a Pesquisa do Estado de Minas Gerais - FAPEMIG.	M. H. is supported by Conselho Nacional de Desenvolvimento Científico e Tecnológico - CNPq.

\end{document}